\documentclass{amsart}
\usepackage{amssymb}
\usepackage{graphicx}
\usepackage{graphics}
\usepackage{amscd}
\usepackage{enumerate}
\usepackage{comment}
\usepackage[all]{xy}
\newtheorem{thm}{Theorem}[section]
\newtheorem{cor}[thm]{Corollary}
\newtheorem{lem}[thm]{Lemma}
\newtheorem{prop}[thm]{Proposition}

\theoremstyle{definition}
\newtheorem{defn}[thm]{Definition}

\newtheorem{qn}[thm]{Question}
\theoremstyle{remark}
\newtheorem{rem}[thm]{Remark}
\newtheorem{ex}[thm]{Example}
\numberwithin{equation}{section}

\newtheorem*{idea}{Idea of the proof}
\newcommand{\spn}[1]{\langle#1\rangle}

\newcommand{\To}{\longrightarrow}
\DeclareMathOperator{\im}{im}
\DeclareMathOperator{\PD}{PD} %Poincar\'e duality
\newcommand{\ol}{\overline}
\newcommand{\Bl}{\text{Bl}}
\newcommand{\SFH}{\mathit{SFH}}
\newcommand{\HFh}{\widehat{\mathit{HF}}}

\newcommand{\HFKh}{\widehat{\mathit{HFK}}}
\newcommand{\CFKh}{\widehat{\mathit{CFK}}}

\newcommand{\de}{\partial}
\newcommand{\CFh}{\widehat{\mathit{CF}}}

\newcommand{\A}{\mathbf{A}}
\newcommand{\gr}{\mathit{gr}}
\newcommand{\agr}{\td\gr}
\renewcommand{\bar}[1]{\overline{#1}}
\renewcommand{\hat}[1]{\widehat{#1}}
\def\a{\alpha}
\def\b{\beta}
\def\g{\gamma}
\def\d{\delta}
\def\e{\varepsilon}
\def\S{\Sigma}

\def\s{\mathfrak{s}}
\def\rs{\s^\circ}

\def\bolda{\boldsymbol{\alpha}}
\def\boldb{\boldsymbol{\beta}}

\def\boldd{\boldsymbol{\delta}}
\def\ab{{\bolda,\boldb}}

\def\W{\mathcal{W}}
\def\Ws{\mathcal{W}^s}
\def\Wb{\mathcal{W}^b}
\def\X{\mathcal{X}}
\def\Int{\text{Int}}
\def\A{\mathcal{A}}
\def\Z{\mathbb{Z}}
\def\Q{\mathbb{Q}}
\def\R{\mathbb{R}}

\def\T{\mathbb{T}}
\def\F{\mathcal{F}}
\def\FF{\mathbb{F}}
\def\CC{\mathcal{C}}

\def\NN{\mathbb{N}}
\def\D{\mathcal{D}}
\def\P{\mathcal{P}}

\def\cT{\mathcal{T}}

\def\L{\mathbb{L}}
\def\H{\mathcal{H}}

\def\BSut{\mathbf{BSut}}

\def\Vect{\mathbf{Vect}_{\mathbb{F}_2}}

\def\x{\mathbf{x}}
\def\y{\mathbf{y}}
\def\spinc{\text{Spin}^c}

\def\bP{\mathbb{P}}
\def\bS{\mathbb{S}}

\def\Id{\text{Id}}
\def\td{\widetilde}
\def\de{\partial}
\def\mb{\mathbb}
\def\mc{\mathcal}
\def\bz{\mathbf{z}}
\newcommand{\vsimeq}{\rotatebox[origin=c]{-90}{\footnotesize $\backsimeq$}}

\begin{document}

\title{Concordance maps in knot Floer homology}%

\author{Andr\'as Juh\'asz}%
\address{Mathematical Institute, University of Oxford, Andrew Wiles Building,
Radcliffe Observatory Quarter, Woodstock Road, Oxford, OX2 6GG, UK}%
\email{juhasza@maths.ox.ac.uk}%
\thanks{AJ was supported by a Royal Society Research Fellowship.}

\author{Marco Marengon}%
\address{Department of Mathematics, Imperial College London,
180 Queen's Gate, London SW7 2AZ, UK}%
\email{m.marengon13@imperial.ac.uk}%
\thanks{MM was supported by an EPSRC Doctoral Training Award.}

\subjclass[2010]{57M27; 57R58}%
\keywords{Concordance; Knot Floer homology; Genus}

\date{\today}%
%\dedicatory{}%
%\commby{}%
% ----------------------------------------------------------------
\begin{abstract}
We show that a decorated knot concordance~$\CC$ from~$K$ to~$K'$ induces a homomorphism~$F_\CC$
on knot Floer homology that preserves the Alexander and Maslov gradings.
Furthermore, it induces a morphism of the spectral sequences to $\HFh(S^3) \cong \Z_2$
that agrees with $F_\CC$ on the $E^1$ page and is the identity on the $E^\infty$ page.
It follows that $F_\CC$ is non-vanishing on $\HFKh_0(K, \tau(K))$.
We also obtain an invariant of slice disks in homology 4-balls bounding~$S^3$.

If~$\CC$ is invertible, then~$F_\CC$ is injective,
hence
\[
\dim \HFKh_j(K,i) \le \dim \HFKh_j(K',i)
\]
for every~$i$, $j \in \Z$.
This implies an unpublished result of Ruberman
that if there is an invertible concordance from the knot~$K$ to~$K'$,
then $g(K) \le g(K')$, where~$g$ denotes the Seifert genus.
Furthermore, if~$g(K) = g(K')$ and~$K'$ is fibred, then so is~$K$.
\end{abstract}

\maketitle
% ----------------------------------------------------------------
\section{Introduction}

Knot Floer homology was introduced independently by Ozsv\'ath-Szab\'o~\cite{OSz3} and Rasmussen~\cite{Ras},
and the first author~\cite{cob} defined maps induced on it by decorated knot cobordisms.
Given a knot~$K$ in~$S^3$, its knot Floer homology with $\Z_2$ coefficients is a finite dimensional bigraded $\Z_2$ vector space
\[
\bigoplus_{i, j \in \Z} \HFKh_j(K,i),
\]
well-defined up to isomorphism, where~$i$ is called the Alexander grading and~$j$ is the homological grading.
The Euler characteristic of $\HFKh_*(K,i)$ is the $i$-th coefficient of the symmetrized Alexander polynomial of~$K$,
and hence knot Floer homology can be viewed as a categorification of the Alexander polynomial.
First, we recall \cite[Definition~4.1]{cob}.

\begin{defn} \label{defn:link}
For $i \in \{0,1\}$, let $Y_i$ be a connected, oriented 3-manifold, and let~ $L_i$ be a non-empty link in~$Y_i$.
Then a \emph{link cobordism} from $(Y_0,L_0)$ to $(Y_1,L_1)$ is a pair~$(X,F)$, where
\begin{enumerate}
\item $X$ is a connected, oriented cobordism from $Y_0$ to $Y_1$,
\item $F$ is a properly embedded, compact, orientable surface in $X$, and
\item $\partial F = L_0 \cup L_1$.
%\item each component of $F$ intersects $Y_1$ non-trivially.
\end{enumerate}
\end{defn}

Knots~$K_0$ and~$K_1$ in~$S^3$ are said to be \emph{concordant} if there is a cobordism~$(X,F)$
from~$(S^3,K_0)$ to~$(S^3,K_1)$ such that~$X = S^3 \times I$ and~$F$ is diffeomorphic to~$S^1 \times I$.
In this case, we call~$(X,F)$ a \emph{concordance} from~$K_0$ to~$K_1$.
In this paper, we also allow more general concordances where $X$
is a cobordism from $S^3$ to $S^3$ such that $H_1(X) = H_2(X) = 0$.

In this paper, a \emph{decorated knot} is a pair $(K,P)$ such that $K$ is a knot,
$P$ is a pair of points in~$K$, and we are given a decomposition of~$K$ into
compact $1$-manifolds~$R_+(P)$ and~$R_-(P)$ such that $R_+(P) \cap R_-(P) = P$.
Given decorated knots $(K_0,P_0)$ and $(K_1,P_1)$ in~$S^3$,
a \emph{decorated concordance} from $(K_0,P_0)$ to $(K_1,P_1)$
is a triple $(X,F,\sigma)$ such that $(X,F)$ is a concordance from $K_0$
to $K_1$, and $\sigma$ consists of two disjoint, properly embedded arcs in~$F$,
one connecting $R_+(K_0)$ and $R_+(K_1)$, the other $R_-(K_0)$ and $R_-(K_1)$.

Dylan Thurston and the first author~\cite{naturality} showed that knot Floer homology is
natural for decorated knots, and Sarkar~\cite{basepoint} proved that moving the basepoints~$P$
around the knot induces a non-trivial automorphism in many cases.
Hence only decorated concordances induce maps on knot Floer homology.

Recall that \cite[Lemma~3.6]{OSz3},
for every decorated knot~$(K,P)$ in~$S^3$, there is a corresponding spectral
sequence
\[
\HFKh(K,P) \Rightarrow \HFh(S^3) \cong \Z_2.
\]
Given an admissible doubly-pointed Heegaard diagram $(\S,\bolda,\boldb,w,z)$ for $(K,P)$,
the singly-pointed diagram $(\S,\bolda,\boldb,w)$ represents $(S^3,w)$,
and~$z$ gives rise to the knot filtration on $\CFh(\S,\bolda,\boldb,w)$.
The spectral sequence arises from this filtered complex.
The $E^0$ page is the associated graded complex $\CFKh(\S,\bolda,\boldb,w,z)$, whose
homology is $\HFKh(K,P)$, the $E^1$ page. The spectral sequence limits
to the homology of~$\CFh(\S,\bolda,\boldb,w)$, which is $\HFh(S^3) \cong \Z_2$.
The filtration level of the generator of~$\Z_2$ in the $E^\infty$ page is the
Ozsv\'ath-Szab\'o $\tau$ invariant~\cite{OSz14}, denoted by $\tau(K)$.

The main result of this paper is that a decorated concordance~$\CC$ induces
a non-vanishing homomorphism~$F_\CC$ on knot Floer homology that preserves the Alexander
and homological gradings, and also induces a morphism of the corresponding spectral
sequences.
The map $F_\CC$ is functorial and depends only on the decorated concordance $\CC$,
while the chain map $f_\CC$ (or even its filtered homotopy type)
need not be functorial, and it can depend on auxiliary data other than $\CC$.

\begin{thm} \label{thm:splitting}
Let $(K_0,P_0)$ and $(K_1,P_1)$ be decorated knots in~$S^3$, and let~$\CC = (X,F,\sigma)$ be a
decorated concordance between them such that $H_1(X) = H_2(X) = 0$.
Then
\[
F_\CC \left(\HFKh_j(K_0,P_0,i) \right) \le \HFKh_j(K_1,P_1,i)
\]
for every $i$, $j \in \Z$.

Furthermore, given an admissible diagram $(\S_r,\bolda_r,\boldb_r,w_r,z_r)$ of $(K_r,P_r)$ for $r \in \{0,1\}$,
there is a filtered chain map
\[
f_\CC \colon \CFh(\S_0,\bolda_0,\boldb_0,w_0) \to \CFh(\S_1,\bolda_1,\boldb_1,w_1)
\]
of homological degree zero such that the induced morphism of spectral sequences agrees with $F_\CC$ on the $E^1$
page and with $\Id_{\Z_2}$ on the total homology and on the $E^\infty$ page.
\end{thm}

Note that the fact that the map induced by a filtered map $f$
on the total homology is an isomorphism in general does not imply that the map $f^\infty$
induced between the $E^\infty$ pages is also an isomorphism. As an example, consider
a complex $C \cong \Z_2$ in filtration level one, and a complex $\bar C \cong \Z_2$
in filtration level zero. If $f \colon C \to \bar C$ is an isomorphism, then $H(f)$
is an isomorphism but $f^\infty$ is not.

In the case of the filtered map $f_\CC$ induced by a decorated concordance $\CC$,
the fact that $f_\CC^\infty$ is an isomorphism follows from the fact that $\tau(K_0) = \tau(K_1)$,
which was shown by Ozsv\'ath and Szab\'o~\cite[Theorem~1.1]{OSz14}.
An alternative proof of this can be given by observing that a decorated concordance
gives filtered maps both ways that induce isomorphisms on the total homology,
as in the proofs of \cite[Theorem~1]{s-invariant} and \cite[Theorem~3.4]{gridtau}.

The invariant $\tau(K)$ can also be defined as the smallest Alexander grading
of an element of $\HFKh(K,P)$ that represents a cycle on each page
of the spectral sequence, and whose homology class in the $E^\infty$ page is~$1$.
We denote the set of such elements by $A_1(K)$. Then we have the following
non-vanishing result for the knot concordance maps:

\begin{cor} \label{cor:non-vanishing}
Let $(K_0,P_0)$ and $(K_1,P_1)$ be decorated knots in~$S^3$, and suppose that $\CC = (X,F,\sigma)$ is a
decorated concordance between them. Let $\tau=\tau(K_0) = \tau(K_1)$. Then, the map
\[
F_\CC \colon \HFKh_0(K_0,P_0,\tau) \to \HFKh_0(K_1,P_1,\tau)
\]
is non-zero, and $F_\CC(A_1(K_0)) \subseteq A_1(K_1)$.
\end{cor}

In fact, for any decorated knot $(K,P)$ in $S^3$, we shall see that
\[
A_1'(K) := A_1(K) \cap \HFKh_0(K,P,\tau(K)) \neq \emptyset,
\]
and the map $F_\CC \colon A_1'(K_0) \to A_1'(K_1)$ is non-zero.

Let $B$ be an integral homology 4-ball with boundary~$S^3$.
Suppose that $S \subset B$ is a slice disk for the decorated knot~$(K,P)$ in~$S^3$.
If we remove a ball from $B$ about a point of $S$, we obtain a concordance~$\CC(S)$
from the unknot $U$ to $K$. By Lemma~\ref{lem:U}, the element
\[
t_{S,P} := F_{\CC(S)}(1) \in \HFKh_0(K,P,0)
\]
is independent of what decoration we choose on~$\CC(S)$.
It is non-zero by Corollary~\ref{cor:non-vanishing},
and is an invariant of the surface~$S$ up to isotopy in $B$ fixing $K$.

\begin{qn}
Can $t_{S,P}$ distinguish different slice disks?
More precisely, is there a decorated knot $(K,P)$ in $S^3$ that has two different slice disks $S$ and $S'$
in $D^4$ such that $t_{S,P} \neq t_{S',P}$?
\end{qn}

Note that, given different decorations $P$ and $P'$ on $K$, the basepoint moving map of Sarkar~\cite{basepoint}
takes $t_{S,P}$ to $t_{S,P'}$, so the answer is independent of the choice of basepoints.

We can use the above viewpoint to refine the approach of Freedman, Gompf, Morrison, and Walker~\cite{mnm}
for disproving the smooth 4-dimensional Poincar\'e conjecture (SPC4).
Suppose that we are given a counterexample to SPC4 with no 3-handles and a single 4-handle.
Removing the 4-handle, we obtain an exotic 4-ball $B$ with boundary homeomorphic to $S^3$.
The belt circles of the 2-handles give a link $L \subset \partial B$, and the cocores
of the 2-handles give a collection of disks $C \subset B$ with boundary~$L$.
If we band sum the components of~$L$ in some way, we obtain a knot~$K \subset \partial B$,
together with a disk $D \subset B$ obtained from $C$.
Hence $D$ induces an element $t_{D,P} \in \HFKh(K,P)$ for any decoration~$P$.
If $t_{D,P} \neq t_{S,P}$ for~$S$ an arbitrary slice disk of $K$, then this implies that $B$ is indeed exotic.

The approach of Freedman et al.~only works if $K$ is not slice in the standard 4-ball,
but it is in the homotopy 4-ball~$B$.
By the work of Ozsv\'ath and Szab\'o~\cite[Theorem~1.1]{OSz14}, the $\tau$-invariant
vanishes if $K$ bounds a disk in a homotopy ball, and so does Rasmussen's $s$-invariant
according to Kronheimer and Mrowka~\cite{KMRas}, so neither can be used for the above purpose.
We could use any other theory equipped with knot concordance maps in manifolds homeomorphic to $S^3 \times I$.
However, note that the Khovanov homology concordance maps of Jacobsson~\cite{Jacobsson}
are only defined when the ambient manifold is diffeomorphic to~$S^3 \times I$.

A knot is called doubly slice if it is a hyperplane cross-section of an unknotted~$S^2$ in~$S^4$.
Motivated by a question of Fox~\cite{Fox} asking which knots are doubly slice,
Sumners~\cite{Sumners} introduced the notion of invertible knot cobordisms.
In his terminology, cobordism stands for concordance; we use the latter for clarity.

\begin{defn}
Let~$K_0$ and~$K_1$ be knots in~$S^3$. We say that a concordance~$(S^3 \times I, F)$ from~$K_0$ to~$K_1$
is \emph{invertible} if there is a concordance~$(S^3 \times I, F')$ from~$K_1$ to~$K_0$ such that the composition
of~$(S^3 \times I, F)$ and~$(S^3 \times I, F')$ from $K_0$ to $K_0$ is equivalent to the trivial cobordism.
We write $K_0 \le K_1$ if there is an invertible cobordism from~$K_0$ to~$K_1$.
\end{defn}

In other words, $F$ is invertible if and only if $(S^3 \times I, F)$ has
a left inverse in the cobordism category of links.
A knot~$K$ is doubly slice if and only if~$U \le K$.
The relation $\le$ is a partial order on the set of knots in~$S^3$,
which follows from~\cite{epi}, as we shall explain later.
%If the decorated concordance~$\CC$ is invertible,
%then~$F_\CC$ is injective, implying Corollary~\ref{cor:main}.

\begin{thm} \label{thm:tech}
If there is an invertible concordance from~$K_0$ to~$K_1$, then
\[
\dim \HFKh_j(K_0,i) \le \dim \HFKh_j(K_1,i)
\]
for every $i$, $j \in \Z$.
\end{thm}

This provides an obstruction to the existence of an invertible concordance
from $K_0$ to $K_1$. According to the work of Manolescu, Ozsv\'ath, and Sarkar~\cite{MOS},
knot Floer homology is algorithmically computable, and Baldwin and Gillam~\cite{computations} used this algorithm to
compute it for knots with at most~12 crossings.

For a knot~$K$ in~$S^3$, we denote its Seifert genus by~$g(K)$.
Ozsv\'ath and Szab\'o~\cite{OSz6} proved that knot Floer homology detects the genus of a knot, in the
sense that
\[
g(K) = \max \{\, i \in \Z \,\colon\, \HFKh_*(K,i) \neq 0 \,\}.
\]
For a simpler proof of this fact, see~\cite{fibred}.
Furthermore, knot Floer homology also detects fibredness of knots,
as $\dim \HFKh_*(K,g(K)) = 1$ if and only if~$K$ is fibred.
This was shown by Ghiggini~\cite{Ghiggini} in the genus one case,
and by Ni~\cite{fibred, corrigendum} and the first author~\cite{decomposition, polytope}
in the general case. These two results, together with Theorem~\ref{thm:tech},
immediately imply the following unpublished result of Ruberman.

\begin{cor} \label{cor:main}
The function~$g$ is monotonic with respect to the partial order~$\le$ induced by invertible concordance.
More concretely, if there is an invertible concordance from~$K_0$ to~$K_1$, then $g(K_0) \le g(K_1)$.
Furthermore, if~$K_1$ is fibred and~$g(K_0) = g(K_1)$, then~$K_0$ is also fibred.
\end{cor}

We now outline a more elementary proof of these results communicated to us by Ruberman,
and which does not use the assumption $g(K_0) = g(K_1)$ for the second statement.
Also see the proof of~\cite[Proposition~3.7]{epi} and the paragraph following it.

\begin{proof}
Let $F$ be an invertible concordance from~$K_0$ to~$K_1$ with inverse~$F'$. Then there is a diffeomorphism
$d \colon S^3 \times I \to S^3 \times I$ such that $d(F' \circ F) = K_0 \times I$ and~$d|_{S^3 \times \partial I}$
is the identity. Let $i \colon S^3 \to S^3 \times I$ be the embedding $i(x) = (x,1/2)$,
and let~$p \colon S^3 \times I \to S^3$ be the projection.
Then the composition
\[
f = p \circ d \circ i \colon S^3 \to S^3
\]
maps~$K_1$ to~$K_0$ such that $f^{-1}(K_0) = K_1$.
We can isotope~$d$ such that $d(K_1 \times \{1/2\})$ becomes transverse
to the $I$-fibration of $K_0 \times I$, and hence $f|_{K_1}$ is an embedding with image $K_0$.
If~$S$ is a minimal genus Seifert surface for~$K_1$,
then $f|_S$ satisfies the conditions of \cite[Corollary~6.23]{Gabai},
hence there exists a Seifert surface~$T$ of~$K_0 = f(K_1)$ such that
$g(T) \le g(S)$. It follows that $g(K_0) \le g(K_1)$.
Recall that \cite[Corollary~6.23]{Gabai} is a deep generalization of
Dehn's lemma to higher genus surfaces due to Gabai. It states that
if~$M$ is a compact oriented 3-manifold, $S$ a compact oriented surface
with connected boundary, and $f \colon S \to M$ a map such that $f|_{\partial S}$
is an embedding and $f^{-1}(f(\partial S)) = \partial S$,
then there exists an embedded surface~$T$ in~$M$ such that $\partial T = f(\partial S)$
and $g(T) \le g(S)$.

Let~$E(K_i)$ denote the exterior of the knot~$K_i$ for~$i \in \{0,1\}$. Then
\[
f|_{E(K_1)} \colon E(K_1) \to E(K_0)
\]
is a degree one map as it is an orientation preserving diffeomorphism between the boundary tori.
%that maps the meridian of $K_1$ to the meridian of~$K_0$ and the longitude of~$K_1$ to the longitude of~$K_0$.
Hence, by \cite[Lemma~1.2]{Rong}, it induces a surjection on the fundamental groups, and also on the commutator subgroups.
If $K_1$ is fibred, then the commutator subgroup $\pi_1(E(K_1))'$ is finitely generated,
hence $\pi_1(E(K_0))'$ is also finitely generated, so $K_0$ is fibred by a result of Stallings~\cite{Stallings}.
\end{proof}

Let $K$ and $K'$ be knots in~$S^3$ such that there is an epimorphism $\pi_1(E(K)) \to \pi_1(E(K'))$
preserving peripheral structure. By~\cite{epi}, this induces a partial order~$\succeq$ on the set of knots.
For example, if there is a degree one map
\[
(E(K),\partial E(K)) \to (E(K'),\partial E(K')),
\]
in particular if $K \ge K'$, then $K \succeq K'$.
Notice that this implies that $\ge$ is also a partial order.
Based on the above proof and Theorem~\ref{thm:tech}, it is natural to ask whether
$K \succeq K'$ also implies that
\begin{equation} \label{eqn:ineq}
\dim \HFKh_*(K,i) \ge \dim \HFKh_*(K',i)
\end{equation}
for every $i \in \Z$.
Note that this would imply \cite[Conjecture~3.6]{epi} claiming that, if~$K \succeq K'$, then $g(K) \ge g(K')$.
Compare this with \cite[Conjecture~9.4]{Lidman}, which claims that if $f \colon Y \to Y'$ is a non-zero
degree map between integer homology spheres, then $\dim \HFh(Y) \ge \dim \HFh(Y')$.
However, inequality~\eqref{eqn:ineq} turns out to be false due to the following example constructed by Jennifer Hom.

\begin{ex} \label{ex:Jen}
Let $K = (T_{2,3})_{2,3}$ be the $(2,3)$-cable of the right-handed trefoil~$T_{2,3}$,
and let $K' = T_{2,3}$. Then $K \succeq K'$.
In fact, there is a degree one map
\[
(E(K),\partial E(K)) \to (E(K'),\partial E(K')).
\]
Indeed, let $T \subset E(K)$ be the boundary of the solid torus used in the satellite
construction for~$K$. Then the exterior of~$T$ is $E(K')$, hence fibred over $S^1$.
If we collapse the fibers to disks, we obtain a degree one map from the exterior of $T$
to~$D^2 \times S^1$, and hence from $E(K)$ to $E(K')$.
But both $K$ and $K'$ are determined by their Alexander polynomials, $K'$ because it is
alternating, and $K$ by the work of Hedden~\cite[Theorem~1.0.6]{Hedden-thesis}.
The symmetrized Alexander polynomial of $K$ is
\[
t^3-t^2+1-t^{-2}+t^{-3},
\]
while the symmetrized Alexander polynomial of $K'$ is $t-1+t^{-1}$. So
$\HFKh(K,1)= 0$ and $\HFKh(K',1)= \Z_2$, violating inequality~\eqref{eqn:ineq}.
\end{ex}

In light of this, we propose the following weaker question.

\begin{qn}
Suppose that $K \succeq K'$. Then is it true that
\[
\dim \HFKh(K) \ge \dim \HFKh(K')?
\]
\end{qn}

The paper is organized as follows. In Section~\ref{sec:cob}, we review sutured manifold cobordisms
and the maps induced by them on sutured Floer homology.
In Section~\ref{sec:knotcob}, we define the knot concordance maps,
show that they preserve the Alexander grading (Proposition~\ref{prop:homogeneous}),
and prove Theorem~\ref{thm:tech}. Section~\ref{sec:SS} gives a brief overview
of spectral sequences arising from a filtered complex.
In Section~\ref{sec:filtration}, we show that, on the chain level,
a knot concordance map can be represented by a chain map that preserves the Alexander
filtration (Theorem~\ref{thm:main1}) and therefore induces a morphism of spectral sequences
(Theorem~\ref{thm:spectral}); this is precisely the second part of Theorem~\ref{thm:splitting}.
Corollary~\ref{cor:non-vanishing} follows from Corollary~\ref{cor:nonvanishing}.
Finally, we prove in Section~\ref{sec:homol} that the knot concordance maps preserve
the homological grading, which concludes the proof of Theorem~\ref{thm:splitting}.

\subsection*{Acknowledgement} We would like to thank Daniel Ruberman for pointing out a more
elementary proof of Corollary~\ref{cor:main}, and Ciprian Manolescu for drawing our attention
to the grading shift formula in~\cite{manolescu2007khovanov} and for his comments on an earlier
version of this paper. We are also grateful to Jennifer Hom for Example~\ref{ex:Jen}.
Finally, we would like to thank the referee for the invaluable suggestions.

\section{Cobordisms of sutured manifolds} \label{sec:cob}

In this section, we briefly review sutured manifold cobordisms,
and the maps they induce on sutured Floer homology,
as defined by the first author~\cite{cob}.

\subsection{Sutured manifolds and sutured cobordisms}

\begin{defn}[{\cite[Definition~2.6]{Gabai}}]
A \emph{sutured manifold} is a compact oriented $3$-manifold $M$ with boundary
together with a set $\gamma \subseteq \de M$ of pairwise disjoint annuli~$A(\gamma)$
and tori~$T(\gamma)$. Furthermore, the interior of each component
of $A(\gamma)$ contains a homologically non-trivial oriented simple closed curve,
called a \emph{suture}. We denote the set of sutures by $s(\gamma)$.

Finally, every component of $R(\gamma)=\de M \setminus \Int(\gamma)$ is oriented
such that $\de R(\gamma)$ is coherent with the sutures. Let $R_+(\gamma)$
(or $R_-(\gamma)$) denote the components of~$R(\gamma)$ whose normal vectors
points out of (into) $M$.
\end{defn}

\begin{defn}[{\cite[Definition~2.2]{sutured}}]
We say that a sutured manifold $(M,\gamma)$ is \emph{balanced} if $M$ has no closed components,
$\chi(R_+(\gamma))=\chi(R_-(\gamma))$, and the map $\pi_0(A(\gamma))\to \pi_0(\de M)$ is surjective.
\end{defn}
%
%\begin{ex}
%\label{ex:Y(L)}
%Let $Y$ be a closed connected oriented $3$-manifold with an unoriented link
%$L$ in $Y$, and let $N(L)$ denote a regular neighborhood  of $L$.
%Then, $Y\setminus N(L)$ with $2$ oppositely oriented meridians on
%each boundary component as sutures is a balanced sutured manifold,
%denoted by $Y(L)$.
%\end{ex}

From now on, we only consider sutured manifolds where $T(\g) = \emptyset$,
and view~$\g$ as a ``thickened'' oriented 1-manifold. So
we often do not distinguish between~$\g$ and~$s(\g)$; it shall be clear from the
context which one we mean.

\begin{defn}[{\cite[Definition~2.3]{cob}}]
\label{def:equivalent}
Let $(M, \gamma)$ be a sutured manifold, and suppose that $\xi_0$
and $\xi_1$ are contact structures on $M$ such that $\de M$ is a
convex surface with dividing set $\gamma$ with respect to both
$\xi_0$ and $\xi_1$. Then we say that $\xi_0$ and $\xi_1$ are
\emph{equivalent} if there is a 1-parameter family
$\left\{\xi_t \,\colon\, t \in I \right\}$ of contact structures
such that $\de M$ is convex with dividing set $\gamma$ with respect
to $\xi_t$ for every $t \in I$. In this case, we write $\xi_0 \sim \xi_1$,
and we denote by $[\xi]$ the equivalence class of the contact structure $\xi$.
\end{defn}

\begin{defn}[{\cite[Definitions~2.4 and~2.14]{cob}}]
Let $(M_0, \gamma_0)$ and $(M_1, \gamma_1)$ be sutured manifolds.
A \emph{cobordism} from $(M_0,\gamma_0)$ to $(M_1, \gamma_1)$ is
a triple $\mc W=(W, Z, [\xi])$, where
\vspace{-5pt}
\begin{itemize}
  \setlength{\itemsep}{1pt}
  \setlength{\parskip}{0pt}
  \setlength{\parsep}{0pt}
\item{$W$ is a compact oriented $4$-manifold with boundary,}
\item{$Z \subseteq \de W$ is a compact, codimension-$0$ submanifold with boundary
(viewed within $\de W$), such that $\de W \setminus \Int(Z) = -M_0 \sqcup M_1$,
and we view $Z$ as a sutured manifold with sutures $\gamma_0 \cup \gamma_1$,}
\item{$\xi$ is a positive contact structure on $Z$ such that $\de Z$ is
a convex surface with dividing set $\gamma_i$ on $\de M_i$ for $i \in \{0,1\}$.}
\end{itemize}
Finally, a cobordism is called \emph{balanced} if both $(M_0, \g_0)$ and $(M_1, \g_1)$ are balanced.
\end{defn}

In this paper, we will only consider balanced sutured manifolds and
balanced cobordisms.

\begin{defn}[{\cite[Definition~2.7]{cob}}]
Two cobordisms $\mc W=(W, Z, [\xi])$ and $\mc W'=(W', Z', [\xi'])$ from $(M_0, \gamma_0)$
to $(M_1, \gamma_1)$ are called \emph{equivalent} if there is an
orientation preserving diffeomorphism $\varphi: W \to W'$ such that $d(Z)=Z'$,
$d_*(\xi)=\xi'$, and $d|_{M_0 \cup M_1}=\Id$.
\end{defn}

\begin{defn}[{\cite[Definition~10.4]{cob}}]
A cobordism $\W = (W, Z, [\xi])$ from $(M_0, \g_0)$ to $(N, \g_1)$
is called a \emph{boundary cobordism} if $W$ is balanced, $N$ is parallel to $M_0 \cup (-Z)$,
and we are also given a deformation retraction $r : W \times \left[0,1\right] \longrightarrow M_0 \cup (-Z)$
such that $r_0|_W=\Id_W$ and $r_1|_N$
is an orientation preserving diffeomorphism from $N$ to $M_0 \cup (-Z)$.
\end{defn}

\begin{defn}[{\cite[Definition~5.1]{cob}}]
We say that a cobordism $\W = (W, Z, [\xi])$ from $(M_0, \g_0)$ to $(M_1, \g_1)$
is \emph{special} if
\begin{enumerate}
\item{$\W$ is balanced;}
\item{$\de M_0 = \de M_1$, and $Z = \de M_0 \times I$ is the trivial cobordism between them;}
\item{$\xi$ is an $I$-invariant contact structure on $Z$ such that each $\de M_0 \times \left\{t\right\}$ is a
convex surface with dividing set $\gamma_0 \times \left\{t\right\}$ for every $t \in I$ with respect to the
contact vector field $\de/\de t$.}
\end{enumerate}
In particular, it follows from (3) that $\g_0 = \g_1$.
\end{defn}

\begin{rem}
\label{rem:cobsplitting}
Every sutured cobordism can be seen as the composition of a boundary
cobordism and a special cobordism, cf.~\cite[Definition~10.1]{cob}.
Let $\W = (W, Z, [\xi])$ be a balanced cobordism from $(M_0, \g_0)$ to
$(M_1, \g_1)$.
%Assume that $Z$ has no \emph{isolated} components, that
%is components that do not intersect $M_1$ (this will always hold for the
%cobordism that we consider in this paper).
Let $(N,\gamma_1)$ be the sutured manifold $(M_0 \cup (-Z), \gamma_1)$.
Then we can think of the cobordism $\W$ as a composition $\Ws \circ \Wb$,
where $\Wb$ is a boundary cobordism from $(M_0, \g_0)$ to $(N, \g_1)$
and $\Ws$ is a special cobordism from $(N,\g_1)$ to $(M_1, \g_1)$.
\end{rem}

\subsection{Relative $\spinc$ structures}

\begin{defn}[{\cite[Definition~3.1]{cob}}]
\label{def:spincMg}
Given a sutured manifold $(M, \gamma)$, we say that a vector field $v$
defined on a subset of $M$ containing $\de M$ is \emph{admissible} if
it is nowhere vanishing, it points into $M$ along $R_-(\g)$, it points
out of $M$ along $R_+(\g)$, and $v|_\g$ is tangent to $\de M$ and either
points into $R_+(\g)$ or is positively tangent to $\g$ (we
think of $\de M$ as a smooth surface, and of $\gamma$ as a 1-manifold).

Let $v$ and $w$ be admissible vector fields on $M$. We say that $v$ and $w$
are \emph{homologous}, and we write $v \sim w$, if there is a collection
of balls $B \subseteq M$, one in each component of $M$, such that $v$ and
$w$ are homotopic on $M \setminus B$ through admissible vector fields.
Then $\spinc(M, \g)$ is the set of homology classes of admissible vector
fields on $M$.
\end{defn}

If $(M,\g)$ is balanced, $\spinc(M,\g)$ is an affine space over $H^2(M,\de M)$.
Throughout this paper, we will denote relative $\spinc$ structures by $\rs$,
to distinguish them from ordinary $\spinc$ structures on oriented
$3$-manifolds, usually denoted by~$\s$.

\begin{rem} \label{rem:v0}
Let $v_0$ be a fixed vector field on $\de M$ arising as $v|_{\de M}$
for some admissible vector field~$v$ on~$M$. We define $\spinc_{v_0}(M,\gamma)$ as the set of
nowhere vanishing vector fields on $M$ that restrict to $v_0$ on $\de M$,
up to isotopy through such vector fields relative to~$\de M$
in the complement of a collection of balls.
Since the space of all possible $v_0$ is contractible, $\spinc_{v_0}(M,\g)$
can be canonically identified with $\spinc(M,\g)$. This was the approach
taken in~\cite{sutured}.
\end{rem}

\begin{defn}[{\cite[Definition~3.2]{cob}}]
Let $(M,\g)$ be a sutured manifold. We say that an oriented 2-plane field~$\xi$ defined on a subset of~$M$
containing~$\partial M$ is \emph{admissible} if
there exists a Riemannian metric~$g$ on~$M$ such that~$\xi^{\perp_g}$ is an admissible vector
field. If~$\xi$ is defined on the whole manifold~$M$, we write
\[
\rs_\xi = [\xi^{\perp_g}] \in \spinc(M,\g).
\]
This is independent of the choice of~$g$ since the space of metrics~$g$ for which
$\xi^{\perp_g}$ is an admissible vector field is convex.
\end{defn}

We now recall the notion of relative $\spinc$ structures
on sutured cobordisms. If~$J$ is an almost complex structure on a 4-manifold
$W$ and $H$ is a 3-dimensional submanifold, then there is a 2-plane field
induced on $H$ called the \emph{field of complex tangencies} along $H$,
cf.~\cite[Lemma 3.4]{cob}.

\begin{defn}[{\cite[Definition~3.5]{cob}}]
\label{def:spincW}
Suppose that $\W = (W, Z, [\xi])$ is a cobordism from the sutured
manifold $(M_0, \g_0)$ to $(M_1, \g_1)$. We say that an almost complex
structure $J$ defined on a subset of $W$ containing $\de Z$ is \emph{admissible}
if the field of complex tangencies on $M_i$ (defined on a subset of $M_i$
containing $\partial M_i$) is admissible
in $(M_i, \g_i)$ for $i \in \left\{0, 1\right\}$, and the field~$\xi_J$ of
complex tangencies on $Z$ (defined on a subset of $Z$ containing $\partial Z$)
is admissible in $(Z, \g_0 \cup \g_1)$.

A \emph{relative $\spinc$ structure} on $\W$ is a homology class of
pairs $(J, P)$, where:
\begin{itemize}
\item{$P \subseteq \Int(W)$ is a finite collection of points,}
\item{$J$ is an admissible almost complex structure defined over $W \setminus P$,}
\item{if $\xi_J$ is the field of complex tangencies along $Z$, then
$\rs_\xi = \rs_{\xi_J}$.}
\end{itemize}
We say that $(J, P)$ and $(J', P')$ are \emph{homologous} if there
exists a compact 1-manifold $C \subseteq W \setminus \de Z$ such that
$P$, $P' \subseteq C$; furthermore, $J|_{W\setminus C}$ and $J'|_{W \setminus C}$
are isotopic through admissible almost complex structures.
We denote by $\spinc(\W)$ the set of relative $\spinc$ structures over~$\W$.
\end{defn}

\begin{rem}
As in the case of sutured manifolds, we will denote relative
$\spinc$ structures on sutured cobordisms by $\rs$, in order
to distinguish them from ordinary $\spinc$ structures on
oriented 4-manifolds, which we denote by $\s$, in analogy with the case
of oriented 3-manifolds.
\end{rem}

\begin{rem}
$\spinc(\W)$ is an affine space over
\[
\ker\left(H^2(W,\de Z) \longrightarrow H^2(Z, \de Z)\right).
\]
There are restriction maps
\[
\spinc(W) \longrightarrow \spinc(M_i, \g_i)
\]
for $i \in \{0,1\}$.
\end{rem}

\subsection{Sutured Floer homology}

The first author~\cite{sutured} associated an $\mb F_2$-vector space $\SFH(M,\g)$
to each balanced sutured manifold $(M, \g)$, called the \emph{sutured
Floer homology} of $(M, \g)$. It splits
along the relative $\spinc$ structures on $(M, \g)$:
\[
\SFH(M, \g) = \bigoplus_{\rs\, \in\, \spinc(M,\g)} \SFH(M, \g, \rs).
\]
Each vector space $\SFH(M, \g, \rs)$ is an invariant of the sutured
manifold together with the relative $\spinc$ structure.
Sutured Floer homology is a common generalisation of Heegaard Floer
homology of closed oriented 3-manifolds~\cite{OSz} and knot Floer homology~\cite{OSz3,Ras}.

The first author proved~\cite{cob} that a balanced cobordism~$\mc W$
from $(M_0, \g_0)$ to $(M_1, \g_1)$ induces a homomorphism
\[
F_\W \colon \SFH(M_0, \g_0) \longrightarrow \SFH(M_1, \g_1).
\]
If $\W$ is endowed with a relative $\spinc$ structure $\rs$, then we also
have a map
\[
F_{\mc W, \rs} \colon \SFH(M_0, \g_0, \rs|_{M_0}) \longrightarrow \SFH(M_1, \g_1, \rs|_{M_1}).
\]

Let $\BSut$ denote the category of balanced sutured manifolds
and equivalence classes of cobordisms, whereas $\Vect$ denotes
the category of vector spaces over~$\mb F_2$.

\begin{thm}[{\cite[Theorem~11.12]{cob}}]
$\SFH$ defines a functor \emph{$\BSut \to \Vect$}, which is a $(3+1)$-dimensional
TQFT in the sense of \emph{\cite{Atiyah}} and \emph{\cite{Blanchet}}.
\end{thm}

We conclude this section by outlining the construction of the
cobordism map associated to a balanced cobordism.
Let $\W=(W, Z, [\xi])$ be a balanced cobordism from $(M_0, \g_0)$ to $(M_1, \g_1)$,
and suppose that every component $Z_0$ of $Z$ intersects $M_1$
(this last hypothesis can actually be dropped, cf.~\cite[Section~10]{cob}).
According to Remark~\ref{rem:cobsplitting}, we can view~$\W$
as the composition of a boundary cobordism $\Wb$ from
$(M_0, \g_0)$ to $(N, \g_1)$ and a special cobordism $\Ws$ from $(N, \g_1)$
to $(M_1, \g_1)$. Using the \emph{contact gluing map} defined by
Honda, Kazez, and Mati\'c~\cite{TQFT}, the first author~\cite[Section~9]{cob}
constructed a map
\[
F_{\Wb} \colon \SFH(M_0, \g_0) \longrightarrow \SFH(N, \g_1)
\]
associated to the special cobordism $\Wb$.

The special cobordism $\Ws$ also induces a map: Choose a
decomposition of $\Ws$ as $\W_3 \circ \W_2 \circ \W_1$, where $\W_i$
is the trace of $i$-handle attachments. Then the first author~\cite{cob}
defined a map $F_{\W_i}$ associated to each cobordism~$\W_i$, and
the map associated to $\Ws$ is defined as
\[
F_{\Ws} = F_{\W_3} \circ F_{\W_2} \circ F_{\W_1} \colon \SFH(N, \g_1) \longrightarrow \SFH(M_1, \g_1).
\]
Finally, the cobordism map $F_\W$ is the composition
$F_{\Ws} \circ F_{\Wb}$, which is independent of all the choices made.

All cobordism maps above admit refinements~$F_{\W, \rs}$
along relative $\spinc$ structures.
The map $F_\W$ can be recovered from the maps $F_{\W, \rs}$ for all $\spinc$
structures~\cite[Definition~10.9 and Proposition~10.11]{cob},
and the $\spinc$ cobordism maps satisfy a type of composition law~\cite[Theorem~11.3]{cob}.

\section{Knot concordance maps} \label{sec:knotcob}

In~\cite{cob}, the first author constructed maps induced on knot Floer homology by decorated link cobordisms.
We recall the necessary definitions, starting with reviewing the real blowup procedure.

\begin{defn}
Suppose that $M$ is a smooth manifold, and let $L \subset M$ be a properly embedded submanifold.
For every $p \in L$, let $N_pL = T_pM/T_pL$ be the fiber of the normal bundle of $L$ over $p$,
and let $UN_pL = (N_pL \setminus \{0\})/\R_+$ be the fiber of the unit normal bundle of $L$ over $p$.
Then the \emph{(spherical) blowup} of $M$ along $L$, denoted by $\Bl_L(M)$, is a manifold with
boundary obtained from $M$ by replacing each point
$p \in L$ by $UN_pL$. There is a natural projection $\Bl_L(M) \to M$. For further details, see
Arone and Kankaanrinta~\cite{AK}.
\end{defn}

We now review decorated links, required to define knot Floer homology
functorially. The following is \cite[Definition~4.4]{cob}.

\begin{defn} \label{defn:declink}
A \emph{decorated link} is a triple $(Y,L,P)$, where $L$ is a non-empty link in the connected oriented 3-manifold $Y$, and $P \subset L$ is a finite set of points. We require that for every component~$L_0$ of~$L$, the number $|L_0 \cap P|$ is positive and even.
Furthermore, we are given a decomposition of~$L$ into compact $1$-manifolds~$R_+(P)$ and~$R_-(P)$ such that $R_+(P) \cap R_-(P) = P$.

We can canonically assign a balanced sutured manifold $Y(L,P) = (M,\g)$ to every decorated link $(Y,L,P)$, as follows.
Let $M = \Bl_L(Y)$ and $\g = \bigcup_{p \in P} UN_pL$. Furthermore,
\[
R_\pm(\g) := \bigcup_{x \in R_\pm(P)} UN_xL,
\]
oriented as $\pm \partial M$, and we orient~$\g$ as~$\partial R_+(\g)$.
\end{defn}

The following is \cite[Definiton~4.2]{cob}.

\begin{defn}
A \emph{surface with divides} $(S,\sigma)$
is a compact orientable surface~$S$, possibly with boundary, together with a properly embedded $1$-manifold~$\sigma$
that divides~$S$ into two compact subsurfaces that meet along~$\sigma$.
\end{defn}

We are now ready to define decorated link cobordisms.
The following is \cite[Definition~4.5]{cob}.

\begin{defn} \label{def:linkcob}
We say that the triple $\X = (X,F,\sigma)$ is a \emph{decorated link cobordism} from $(Y_0,L_0,P_0)$ to $(Y_1,L_1,P_1)$ if
\begin{enumerate}
\item $(X,F)$ is a link cobordism from $(Y_0,L_0)$ to $(Y_1,L_1)$,
\item $(F,\sigma)$ is a surface with divides such that the map
\[
\pi_0(\partial\sigma) \to \pi_0((L_0 \setminus P_0) \cup (L_1 \setminus P_1))
\]
is a bijection,
\item \label{it:crossing} we can orient each component~$R$ of~$F \setminus \sigma$ such that whenever~$\partial\ol{R}$
crosses a point of~$P_0$, it goes from~$R_+(P_0)$ to~$R_-(P_0)$,
and whenever it crosses a point of~$P_1$, it goes from~$R_-(P_1)$ to~$R_+(P_1)$,
\item if $F_0$ is a closed component of $F$, then $\sigma \cap F_0 \neq \emptyset$.
\end{enumerate}

%Two decorated link cobordisms $\X = (X,F,\sigma)$ and $\X' = (X',F',\sigma')$ between
%the same decorated links $(Y_0,L_0,P_0)$ and $(Y_1,L_1,P_1)$ are said to be \emph{equivalent}
%if there is an orientation preserving diffeomorphism $d \colon X \to X'$ such that $d(F) = F'$ and~$d(\sigma) = \sigma'$;
%moreover, $d(y) = y$ for every $y \in Y_0 \cup Y_1$.
%
%Suppose that $\X = (X,F,\sigma)$ is a cobordism from $(Y_0,L_0,P_0)$ to $(Y_1,L_1,P_1)$ and let $\X' = (X',F',\sigma')$ be a
%cobordism from $(Y_0',L_0',P_0')$ to $(Y_1',L_1',P_1')$.
%We say that~$\X$ and~$\X'$ are \emph{diffeomorphic} if there exists an orientation preserving
%diffeomorphism~$d$ from~$X$ to~$X'$ such that $d(F) = F'$ and $d(\sigma) = \sigma'$,
%and $d(R_\pm(P_i)) = R_\pm(P_i')$ for~$i \in \{0,1\}$.
%
%Decorated links and equivalence classes of decorated link cobordisms form a category $\link$
%with the obvious composition and identity morphisms.
%As each link component has at least two marked points, when composing two
%decorated link cobordisms, we do not create undecorated closed components of the surface.
\end{defn}

Finally, we recall how to associate a sutured manifold cobordism complementary to a decorated link cobordism.
For this purpose, we first discuss $S^1$-invariant contact structures on circle bundles;
see also~\cite[Section 4]{cob}.
Let $\pi \colon M \to F$ be a principal circle bundle over a compact oriented surface $F$.
An $S^1$-invariant contact structure $\xi$ on $M$ determines a diving set $\sigma$ on the
base $F$, by requiring that $x \in \sigma$ if and only if $\xi$ is tangent to $\pi^{-1}(x)$,
and a splitting of $F$ as $R_+(\sigma) \cup R_-(\sigma)$. The image of any local
section of~$\pi$ is a convex surface with dividing set projecting onto~$\sigma$.
According to Lutz~\cite{Lutz} and Honda~\cite[Theorem~2.11 and Section~4]{Ko},
given a dividing set $\sigma$ on $F$ that intersects each component of $F$
non-trivially and divides $F$ into subsurfaces $R_+(\sigma)$ and $R_-(\sigma)$,
there is a unique $S^1$-invariant contact structure~$\xi_\sigma$ on $M$, up to isotopy,
such that the dividing set associated to $\xi_\sigma$ is exactly~$\sigma$,
the coorientation of~$\xi_\sigma$ induces the splitting~$R_\pm(\sigma)$, and the
boundary~$\de M$ is a convex.

The following is~\cite[Definition~4.9]{cob}.

\begin{defn} \label{defn:W}
Let $(X,F,\sigma)$ be a decorated link cobordism from $(Y_0,L_0,P_0)$ to~$(Y_1,L_1,P_1)$.
Then we define the sutured cobordism $\W = \W(X,F,\sigma)$ as follows.
Choose an arbitrary splitting of~$F$ into~$R_+(\sigma)$ and~$R_-(\sigma)$ such that $R_+(\sigma) \cap R_-(\sigma) = \sigma$,
and orient~$F$ such that~$\partial R_+(\sigma)$ (with $R_+(\sigma)$ oriented as a subsurface of $F$)
crosses~$P_0$ from~$R_+(P_0)$ to~$R_-(P_0)$ and~$P_1$ from~$R_-(P_1)$ to~$R_+(P_1)$.
Then~$\W$ is defined to be the triple $(W,Z,[\xi])$, where $W = \Bl_F(X)$ and $Z = UNF$, oriented
as a submanifold of~$\partial W$, finally $\xi = \xi_{\sigma}$ is an $S^1$-invariant contact structure
with dividing set~$\sigma$ on~$F$ and convex boundary~$\partial Z$
with dividing set projecting to~$P_0 \cup P_1$.
%
%Note that $\W$ is a cobordism from $\W(Y_0,L_0,P_0) = (M_0,\g_0)$ to $\W(Y_1,L_1,P_1) = (M_1,\g_1)$.
%Indeed, let $\pi \colon Z \to F$ be the natural projection, then we can assume that~$\pi(\g_\sigma) = P_0 \cup P_1$.
%By \cite[Lemma~4.8]{cob} and our assumptions above,
%\[
%\pi(R_\pm(\g_\sigma)) \cap L_0 = R_\pm(P_0) \text{ and } \pi(R_\pm(\g_\sigma)) \cap L_1 = R_\mp(P_1).
%\]
%This implies that $R_\pm(\g_\sigma) \cap \partial M_0 = R_\pm(\g_0)$ and $R_\pm(\g_\sigma) \cap \partial M_1 = R_\mp(\g_1)$.
%Since $\partial Z = \partial M_0 \cup (-\partial M_1)$, we obtain that~$\g_\sigma \cap \partial M_0 = \g_0$
%and $\g_\sigma \cap \partial M_1 = \g_1$.
\end{defn}

The contact vector fields with respect to which a local section of $UNF \to F$
and~$\de Z$ are transverse are different, so they can project to different subsets
of $L_0 \cup L_1$. Specifically, we have that the dividing set for $\de Z$ projects
to $P_0 \cup P_1$, while~$\de \sigma$ is disjoint from $P_0 \cup P_1$.

Notice that if $F$ does not have any closed component, then it deformation retracts
onto a 1-dimensional CW complex, and therefore any $S^1$-bundle on it has a section,
hence is trivial if the bundle is orientable. In particular, $UNF \approx F \times S^1$.

In the present paper, we only consider decorated links~$(Y,L,P)$ where~$Y = S^3$, the link $L$ has a single component,
and~$|P| = 2$. Hence, we drop~$Y$ from the notation and only write~$(K,P)$ for such a decorated knot.

\begin{defn}
A \emph{decorated concordance} is a decorated link cobordism $(X,F,\sigma)$ such that
\begin{enumerate}
\item $X$ is an integer homology $S^3 \times I$ with boundary $(-S^3) \sqcup S^3$,
\item the surface~$F$ is an annulus, and
\item $\sigma$ consists of two arcs connecting the two components of~$\partial F$.
\end{enumerate}
\end{defn}

If $X = S^3 \times I$, we drop~$X$ from the notation and only write~$(F,\sigma)$.

\begin{lem}
Let $X$ be an oriented cobordism from $S^3$ to $S^3$. Then $X$ has the same homology and
cohomology as $S^3 \times I$ if and only if $H_1(X) = H_2(X) = 0$.
\end{lem}

\begin{proof}
The ``only if'' part is obvious. So suppose that $H_1(X) = H_2(X) = 0$.
Then let $\ol X$ be the closed 4-manifold obtained by gluing two 4-balls to $\partial X$.
We denote by $B \subset X$ the union of these 4-balls. Then, for $i \in \{1,2\}$, we have
\[
0 = H_i(X) \cong H^{4-i}(X,\partial X) \cong H^{4-i}(\ol X, B) \cong H^{4-i}(\ol X).
\]
Here, the first isomorphism follows from Poincar\'e-Lefschetz duality, the second from excision,
and the third from the cohomological long exact sequence of the pair $(\ol X, B)$.
So $H^2(\ol X) = H^3(\ol X) = 0$, hence $H_1(\ol X) \cong H^3(\ol X) = 0$,
and $H^1(\ol X) = \text{Hom}(H_1(\ol X),\Z) = 0$. As $\ol X$ has the same integral cohomology is $S^4$,
after removing two balls, $X$ has the same integral homology and cohomology as $S^3 \times I$.
\end{proof}

It follows from \cite[Proposition~4.10]{cob} that a decorated concordance~$\CC = (X,F,\sigma)$
from~$(K_0,P_0)$ to~$(K_1,P_1)$ induces a homomorphism
\[
F_{\CC} \colon \HFKh(K_0,P_0) \to \HFKh(K_1,P_1),
\]
where $\HFKh(K_i,P_i)$ are the natural knot Floer homology groups defined in~\cite{naturality}.
Indeed, $\W = \W(X,F,\sigma)$ is a cobordism from the sutured manifold $S^3(K_0,P_0)$ to $S^3(K_1,P_1)$,
and hence induces a homomorphism
\[
F_\W \colon \SFH(S^3(K_0,P_0)) \to \SFH(S^3(K_1,P_1)).
\]
But $\SFH(S^3(K_0,P_0)) \cong \HFKh(K_0,P_0)$ and $\SFH(S^3(K_1,P_1)) \cong \HFKh(K_1,P_1)$ tautologically.
This assignment is functorial under composition of link cobordisms.

\subsection{Relative $\spinc$ structures and knot concordances}

In the case of knot concordances, the relative $\spinc$ structures
behave nicely, as explained in this section.

\begin{lem}
\label{lem:easyspinc}
Suppose that $\CC = (X,F, \sigma)$ be a decorated concordance from $(K_0, P_0)$ to $(K_1, P_1)$.
If $(M_i, \g_i) = S^3(K_i, P_i)$ is the balanced sutured manifold complementary to $(K_i, P_i)$ for
$i \in \{0,1\}$, and $\W =  \W(\CC) = (W, Z, [\xi])$ is the sutured manifold cobordism from
$(M_0, \g_0)$ to $(M_1, \g_1)$ complementary to~$\CC$, then
\begin{equation}
\label{eqn:split}
F_\W = \bigoplus_{\rs\,\in\,\spinc(\W)} F_{\W,\rs}.
\end{equation}
Furthermore, $\spinc(\W)$ is an affine space over $H^2(W, Z) \cong \Z$,
and the restriction maps
\[
r_i \colon \spinc(\W) \To \spinc(M_i, \g_i)
\]
are isomorphisms for $i \in \{0,1\}$.
\end{lem}

\begin{proof}
As in Remark~\ref{rem:cobsplitting}, we write~$\W = \Ws \circ \Wb$, where~$\Wb$ is a boundary cobordism
from~$(M_0,\g_0)$ to~$(N,\g_1)$, where $N = M_0 \cup (-Z)$, and~$\Ws$ is a special cobordism from
$(N,\g_1)$ to $(M_1,\g_1)$.
As~$Z$ is a product, $N$ is diffeomorphic to the knot complement~$M_0 \approx S^3 \setminus N(K_0)$,
and hence~$H_2(N) = 0$. So, by~\cite[Remark~10.10]{cob} and~\cite[Proposition~10.11]{cob},
\[
F_\W = \bigoplus_{\rs\,\in\,\spinc(\W)} F_{\W,\rs}.
\]
%Furthermore, $\spinc(\W_0) = \spinc(M_0,\g_0)$, and for every $\ft_0 \in \spinc(M_0,\g_0)$
%there is a corresponding element $\s_0 \in \spinc(M_0,\g_0)$ that restricts to $f_{-\xi}(\ft_0)$.

As $H^k(Z,\partial M_1) = 0$ for $k \in \{1,2\}$, we can apply \cite[Lemma~3.7]{cob} to
conclude that
\[
\spinc(\W) \cong H^2(W,\partial M_1).
\]
Of course, $H^2(W,\partial M_1) \cong H^2(W,\partial M_0) \cong H^2(W,Z)$.
By excision, $H^2(W,Z) \cong H^2(X,N(F))$, where $N(F)$ is a regular neighborhood  of~$F$.
From the long exact sequence of the pair $(X,N(F))$, the fact that $H^1(X) = H^2(X) = 0$,
and since  $H^1(N(F)) \cong H^1(S^1) \cong \Z$, we obtain that $H^2(X,N(F)) \cong \Z$.

The restriction maps
\[
r_i \colon \spinc(\W) \to \spinc(M_i,\g_i)
\]
for $i \in \{0,1\}$ are modelled on
the restriction maps
$H^2(W,\partial M_i) \to H^2(M_i,\partial M_i)$ for $i \in \{0,1\}$.
From the long exact sequence of the triple $(W,M_i,\partial M_i)$, the sequence
\begin{equation} \label{eqn:triple}
H^2(W,M_i) \to H^2(W,\partial M_i) \to H^2(M_i, \partial M_i) \to H^3(W,M_i)
\end{equation}
is exact.
Now consider the relative Mayer-Vietoris sequence of the pairs $(W,M_i)$ and $(N(F),N(K_i))$,
whose union is $(X, \partial_i X)$, where $\partial_i X \approx S^3$ is the ingoing boundary component of~$X$
when $i = 0$, and is the outgoing boundary component when $i = 1$:
\[
H^k(X, \partial_i X) \to H^k(W,M_i) \oplus H^k(N(F),N(K_i)) \to H^k(Z, \partial M_i).
\]
Here, $H^k(X, \partial_i X) \cong H^k(S^3 \times I, S^3 \times \{0\}) = 0$,
and the last term is zero as~$Z$ deformation retracts onto~$\partial M_i$.
Consequently, $H^k(W,M_i) = 0$ for every~$k$, and by the exact sequence~\eqref{eqn:triple},
this means that the restriction maps~$r_i$
are isomorphisms for $i \in \{0,1\}$.
\end{proof}

In the following lemma, $v_0$ denotes any fixed vector field on a balanced sutured
manifold $(M, \g)$ obtained by restricting an admissible vector field to $\de M$,
cf.~Definition~\ref{def:spincMg} and Remark~\ref{rem:v0}.

\begin{lem}
\label{lem:samespinc}
Let $\CC = (X,F, \sigma)$ be a knot concordance from $(K_0, P_0)$ to $(K_1, P_1)$.
As in Lemma~\ref{lem:easyspinc}, let $(M_i, \g_i) = S^3(K_i, P_i)$ for $i \in \{0,1\}$,
and let
\[
\W = \W(\CC) = (W, Z, [\xi]).
\]
For $i \in \{0,1\}$, let $S_i$ be a Seifert surface for $K_i$, and let $t_i$ be the trivialization
of $v_0^\perp$ given by a vector field tangent to $\de M_i$ in the meridional
direction. Then, for any relative $\spinc$ structure $\rs \in \spinc(\W)$,
\begin{equation}
\label{eqn:chern}
\langle\, c_1(r_0(\rs), t_0) , [S_0] \,\rangle
=
\langle\, c_1(r_1(\rs), t_1) , [S_1] \,\rangle,
\end{equation}
where $r_0$ and $r_1$ are the restriction maps in Lemma~\ref{lem:easyspinc}.
\end{lem}

From Lemma~\ref{lem:samespinc}, we can already deduce the following
proposition, which can be seen as a first step towards the proof of
Theorem~\ref{thm:splitting}.

\begin{prop}
\label{prop:homogeneous}
If $\CC$ is a decorated concordance between two knots $(K_0,P_0)$
and $(K_1, P_1)$, then the map induced between the knot Floer
homologies preserves the Alexander grading, that is
\[
F_\CC \left( \HFKh(K_0, P_0, i) \right) \leq \HFKh(K_1, P_1, i)
\]
for every $i \in \Z$.
\end{prop}

\begin{proof}
We use the same notation as in Lemmas~\ref{lem:easyspinc} and~\ref{lem:samespinc}.
It follows from Lemma~\ref{lem:easyspinc} that the map $F_\CC = F_\W$
splits as the sum of the maps $F_{\W, \rs}$ for $\rs \in \spinc(\W)$, cf.~Equation~\eqref{eqn:split}.
It is therefore sufficient to check that, for every relative $\spinc$ structure
$\rs\in\spinc(\W)$, the map
\[
F_{\W, \rs} \colon \SFH(M_0, \g_0, \rs|_{M_0}) \To \SFH(M_1, \g_1, \rs|_{M_1})
\]
preserves the Alexander grading.

According to the proof of~\cite[Theorem~1.5]{decomposition} on page~27,
if $t_i$ is the trivialization of~$v_0^\perp$ given by a vector field tangent
to $\partial M_i$ in the meridional direction, then
\[
\SFH(M_i,\g_i,\rs) = \HFKh(K_i,P_i, -\langle\, c_1(\rs,t_i) , [S_i] \,\rangle / 2),
\]
where $S_i$ is a Seifert surface of $K_i$ for $i \in \{0,1\}$. The result now follows
from Lemma~\ref{lem:samespinc}, which states that
\[
\langle\, c_1(\rs|_{M_0}, t_0) , [S_0] \,\rangle
=
\langle\, c_1(\rs|_{M_1}, t_1) , [S_1] \,\rangle.
\qedhere
\]
\end{proof}

\begin{proof}[Proof of Lemma~\ref{lem:samespinc}]
Choose an admissible almost complex structure~$J$ on~$W \setminus P$
whose homology class is~$\rs$,
where $P \subset \Int(W)$ is a finite set of points, as in Definition~\ref{def:spincW}.
Let~$\xi_J$ be the field of complex tangencies of~$J$ along~$Z$. Then, by definition,
$\rs_\xi = \rs_{\xi_J}$. In fact, we can choose~$J$ such that $\xi_J = \xi$.
%By Lemma~\ref{lem:seif}, $(S^3 \times I, F)$ has a distinguished Seifert framing, giving rise to
Choose a trivialization of the normal $S^1$ bundle of~$F$ whose total space is~$Z$.
If we identify~$F$ with $S^1 \times I$ such that $\sigma$ maps to~$P_0 \times I$ for~$P_0 = \sigma \cap K_0$,
then this identification, together with the above trivialization,
induce a diffeomorphism $d \colon Z \to S^1 \times S^1 \times I$, where the
first factor is the fiber direction, and such that~$\xi$ is mapped to
an $I$-invariant contact structure with dividing set $S^1 \times P_0 \times \{a\}$ on $S^1 \times S^1 \times \{a\}$ for
every $a \in I$, and $\{\theta\} \times P_0 \times I$ on $\{\theta\} \times S^1 \times I$ for every $\theta \in S^1$.
Hence, we can perturb the 2-plane field~$\xi$ such that it is always tangent to the second~$S^1$
factor; i.e., the longitudinal direction. So we can choose~$J$ such that~$\xi_J$
is also invariant in the $\sigma$ direction, and it contains the longitude direction.
If~$v$ is a nowhere zero section of~$\xi_J$ tangent to the longitude direction,
then -- under a homotopy of~$\xi_J|_{\partial M_i}$ to~$v_0^\perp$ through admissible 2-plane fields --
the vector field $v|_{\partial M_0}$ represents a trivialization~$\tau_0$ that corresponds to~$t_0$
and $v|_{\partial M_1}$ represents a trivialization~$\tau_1$ that corresponds to~$t_1$.

The 2-plane field~$\xi_J$, together with the trivialization given by~$v$, gives a complex 1-dimensional subbundle
of $(TW|_Z,J)$ together with a trivialization. The complement of $\xi_J$ is also trivial, canonically trivialized
by its intersection with~$TZ$, which then gives rise to a trivialization~$\tau$ of $TW|_Z$.
As~$J$ is defined over the 3-skeleton of~$W$, it makes sense to talk about the relative Chern class~$c_1(TW,J,\tau) \in H^2(W,Z)$.
If~$\xi^i_J$ denotes the field of complex tangencies of~$J$ along~$M_i$, then the complement of~$\xi^i_J$
is a trivial bundle (trivialized by its intersection with~$TM_i$), so
\[
c_1(\xi^i_J, \tau_i) = c_1(TW|_{M_i},J,\tau) =  c_1(TW,J,\tau)|_{M_i},
\]
where the second equality follows from the naturality of Chern classes.
By construction, $\xi^i_J$ represents~$\rs_i$.

%By Lemma~\ref{lem:seif}, there is a Seifert cobordism~$V$ for~$(S^3 \times I, F)$ such that $S_i = V \cap M_i$ is a Seifert
%surface of~$K_i$ for $i \in \{0,1\}$. Then we can view~$V$ as a submanifold of~$W$ that represents
%a homology between $[V] \in H_3(W,Z)$.
Recall that $S_i$ is a Seifert surface of~$K_i$ for~$i \in \{0,1\}$.
Note that $H_2(W,Z) \cong \Z$, and that there is a bilinear intersection pairing
\[
H_2(W,Z) \otimes H_2(W,M_0 \cup M_1) \to \Z.
\]
Consider the cycle $m = S^1 \times \{\text{pt}\} \times I$ in $C_2(W,M_0 \cup M_1)$.
As both~$S_0$ and~$S_1$ intersect~$m$ once positively, they both represent the
generator of $H_2(W,Z) \cong \Z$. Hence
\[
\langle\, c_1(\rs_0,\tau_0) , [S_0] \,\rangle = \langle\, c_1(TW, J,\tau) , [S_0] \,\rangle =
\langle\, c_1(TW, J,\tau) , [S_1] \,\rangle = \langle\, c_1(\rs_1,\tau_1) , [S_1] \,\rangle,
\]
and Equation~\eqref{eqn:chern} follows as we saw that~$\tau_0$ corresponds to~$t_0$ and~$\tau_1$ corresponds to~$t_1$.
\end{proof}

As a consequence of Proposition~\ref{prop:homogeneous}, we can prove Theorem~\ref{thm:tech}.

\begin{proof}[Proof of Theorem~\ref{thm:tech}]
Suppose that~$F$ is an invertible concordance from~$K_0$ to~$K_1$. Choose an arbitrary pair of points~$P_0$
on~$K_0$ and~$P_1$ on~$K_1$, making them into decorated knots, and an arbitrary pair of arcs~$\sigma$ on~$F$
making~$F$ into a decorated concordance from~$(K_0,P_0)$ to~$(K_1,P_1)$. Let~$F'$ be the inverse of~$F$,
and choose a decoration~$\sigma'$ on it such that $(F',\sigma')$ is a decorated concordance from~$(K_1,P_1)$
to~$(K_0,P_0)$. As the composition of~$F$ and~$F'$ is equivalent to the trivial cobordism~$K_0 \times I$
from~$K_0$ to~$K_0$, we can choose~$\sigma'$ such that the composition of~$\CC = (F,\sigma)$ and~$\CC' = (F',\sigma')$
is equivalent to the product decorated cobordism~$(K_0 \times I, P \times I)$, where $P = \sigma \cap K_0$
is a pair of points. By the functoriality of $F_\CC$ and the fact that a product cobordism induces the identity map,
\[
F_{\CC'} \circ F_{\CC} = \Id_{\HFKh(K_0,P_0)},
\]
and so $F_{\CC}$ is injective. We shall see in Section~\ref{sec:homol} that $F_\CC$
preserves the homological grading.
Hence Proposition~\ref{prop:homogeneous} implies that
\[
\dim \HFKh_j(K_0,P_0,i) \le \dim \HFKh_j(K_1,P_1,i)
\]
for every $i$, $j \in \Z$. Up to isomorphism, $\HFKh_j(K_i,P_i)$ is independent of the choice of~$P_i$,
and the result follows.
\end{proof}

We shall see in Section~\ref{sec:homol} that the concordance maps also preserve the
homological grading. Then we have the following.

\begin{lem} \label{lem:U}
Let $\CC = (X,F,\sigma)$ be a decorated concordance from $(K_0,P_0)$ to $(K_1,P_1)$.
If $K_0$ is the unknot~$U$, then the element
\[
F_\CC(1) \in \HFKh_0(K_1,P_1,0)
\]
is independent of the decorations~$\sigma$ and $P_0$,
where $1 \in \HFKh(K_0,P_0) \cong \Z_2$.
\end{lem}

\begin{proof}
Suppose that $\sigma'$ be another decoration with the same endpoints as $\sigma$,
let $\CC' = (X,F,\sigma')$, and define
\[
k = [\sigma' - \sigma] \in H_1(F) \cong \Z.
\]
Consider the decorated concordance $\CC_k = (S^3 \times I, U \times I, \sigma_k)$,
where $\sigma_k$ spirals around $k$ times.
Then $\CC' = \CC \circ \CC_k$.
As $\HFKh(U) \cong \Z_2$,
we have $F_{\CC_k} = \Id_{\Z_2}$.
By the functoriality of the knot concordance maps,
we obtain that $F_{\CC'} = F_\CC$.
Since $\HFKh(U) \cong \Z_2$ has no non-trivial
automorphisms, it does not matter
how we choose the markings~$P_0$.
\end{proof}

\section{Filtered complexes and spectral sequences} \label{sec:SS}

In this section, we briefly recall the definitions and properties of
spectral sequences that we need. We mainly refer to the book of McCleary~\cite{mccleary}.
The spectral sequences we are interested in arise from filtered chain
complexes, so we focus on this case only.

\begin{defn}
A \emph{filtered chain complex} is a chain complex $(C = \bigoplus_{k \in \Z} C_k,\de)$,
such that $\de C_k \subseteq C_{k-1}$, with a nested sequence of subcomplexes
\[\ldots \subseteq \F_{p-1}C \subseteq \F_pC \subseteq \F_{p+1}C \subseteq \ldots\]
such that $\bigcup_{p \in \Z} \F_p C = C$ and $\de(\F_p C) \subseteq \F_p C$.

We say that the filtered chain complex is \emph{bounded} if there are
integers $a \le b$ such that
\[\left\{0\right\} = \F_a C\subseteq \ldots \subseteq \F_b C =C.\]
\end{defn}

We obtain a spectral sequence from a filtered chain complex as follows,
cf.~\cite[Proof of Theorem~2.6]{mccleary}.

\begin{defn}
For $p$, $q$, $r \in \Z$, we define
\begin{equation*}
\begin{split}
Z_{p,q}^r &=  \F_p C_{p+q} \cap \de^{-1}\left( \F_{p-r} C_{p+q-1}\right),\\
B_{p,q}^r &=  \F_p C_{p+q} \cap \de\left( \F_{p+r} C_{p+q+1}\right),\\
Z_{p,q}^\infty &=  \F_p C_{p+q} \cap \ker \de,\\
B_{p,q}^\infty &=  \F_p C_{p+q} \cap \im \de.
\end{split}
\end{equation*}
For $0 \leq r\leq \infty$, the \emph{$r$-page} (or \emph{$r$-term})
is the complex $\left(E^r = \bigoplus_{p,q \in \Z} E_{p,q}^r, \de^r \right)$, where
\[E_{p,q}^r = \frac{Z_{p,q}^r}{Z_{p-1,q+1}^{r-1} + B_{p,q}^{r-1}},\]
and the differential
\[\de^r:E_{p,q}^r \to E_{p-r,q+r-1}^r\]
is induced by the differential $\de$ on the complex $C$.
\end{defn}

Sometimes we only focus on the $p$ grading. In such cases,
we drop $q$ from the notation, and write $E^r_p = \bigoplus_{q \in \Z} E^r_{p,q}$.
For the following, see~\cite[Proof of Theorem~2.6]{mccleary}.

\begin{thm}
\label{thm:spectralsequence}
The pages $\left\{(E^r,\de^r)\right\}$ induced by a filtered chain complex
form a spectral sequence in the sense of \cite[Definition~2.2]{mccleary}; i.e.,
\[
E_{p,q}^{r+1} = H_{p,q}(E^r_{*,*}, \de^r) := \frac{\ker\left(\de^r |_{E_{p,q}^r}\right)}{\im\left(\de^r |_{E_{p+r, q-r+1}^r}\right)}.
\]
If the filtration is bounded, then there is a canonical isomorphism
\[E_{p,q}^\infty \cong \frac{\F_p \left(H_{p+q}(C)\right)}{\F_{p-1} \left(H_{p+q}(C)\right)},\]
where the filtration on the total homology $H(C) = \bigoplus_{k \in \Z} H_k(C)$ is the one induced from $C$:
\[
\F_p(H(C)) := \im \left( H(\F_pC, \de|_{\F_pC}) \to H(C,\de) \right).
\]
\end{thm}

\begin{rem}
Notice that $E^0_{p,q}$ is the graded module
\[\frac{\F_pC_{p+q}}{\F_{p-1}C_{p+q}}\]
associated with the filtration.
The page $E^1_{p,q}$ is the homology $H_q(E^0_{p,*}, \de^0)$ of the associated
graded module with the induced differential.
\end{rem}

\subsection{Morphisms of spectral sequences}

According to McCleary~\cite{mccleary}, we have the following.

\begin{defn}
\label{def:morphism}
Let ${\left(E^r, \de^r\right)}$ and ${\left(\bar E^r, \bar \de^r\right)}$ be spectral sequences.
A \emph{morphism of spectral sequences} is a sequence module homomorphisms
$f^r: E^r_{*,*} \to \bar E^r_{*,*}$ for $r \in \NN$, of bidegree $(0,0)$,
such that $f^r$ commutes with the differentials; i.e., $f^r \circ \de^r = \bar \de^r \circ f^r$,
and each $f^{r+1}$ is induced by $f^r$ on homology; i.e., $f^{r+1}$ is the composite
\[f^{r+1}: E_{*,*}^{r+1} \cong H\left(E_{*,*}^r, \de^r\right) \xrightarrow{H\left(f^r\right)}
H\left(\bar E_{*,*}^r, \bar \de^r\right) \cong \bar E_{*,*}^{r+1}.\]
\end{defn}

\begin{rem}
\label{rem:morphss}
Let $f:C \to \bar C$ be a map of filtered complexes of homological degree zero; i.e.,
\begin{itemize}
  \item{$f(C_k) \subseteq \bar C_k$,}
	\item{$f \circ \de = \bar \de \circ f$,}
	\item{$f\left(\F_pC\right) \subseteq \F_p{\bar C}$.}
\end{itemize}
Then $f$ induces a morphism between the spectral sequences associated
to~$C$ and~$\bar C$.
\end{rem}

\begin{rem}
If ${\left(E^r, \de^r\right)}$ and ${\left(\bar E^r, \bar \de^r\right)}$
are bounded spectral sequences,
\[
\left\{f^r: E^r \to \bar E^r\right\}
\]
is a morphism of spectral sequences,
and~$f^\infty$ is non-zero on~$E_{p,q}^\infty$, then $f^r$ is non-zero on $E_{p,q}^r$
for all $r \in \NN$.
\end{rem}

\subsection{The $\tau$ invariant}

In this subsection, we recall the definition and few properties of the
Ozsv\'ath-Szab\'o $\tau$ invariant, and we discuss it in a slightly more
general setting.

\begin{defn}
\label{def:tauinvariant}
If $C$ is a non-acyclic bounded filtered complex over~$\FF_2$, we define
\[
\tau(C) := \min\left\{ p \in \Z \,\colon\, H(\F_pC) \to H(C) \,\mbox{is non-trivial} \right\}.
\]
\end{defn}

Definition~\ref{def:tauinvariant} generalizes the Ozsv\'ath-Szab\'o $\tau$
invariant in the sense that, if $C = \CFh(\H)$ for some Heegaard diagram for
a decorated knot $(K,P)$, then $\tau(C) = \tau(K)$.

\begin{rem}
An alternative definition of $\tau(C)$ is given by the following property:
\begin{equation*}
E^\infty_p(C)
\,\,\,
\begin{cases}
= 0 & \mbox{if $p < \tau(C)$} \\
\not= 0 & \mbox{if $p = \tau(C)$.}
\end{cases}
\end{equation*}
Furthermore, if the total homology $H(C) = \mb F_2$, then
\begin{equation*}
E^\infty_p(C)
\,\,\,
\begin{cases}
= 0 & \mbox{if $p \not= \tau(C)$} \\
\not= 0 & \mbox{if $p = \tau(C)$.}
\end{cases}
\end{equation*}
\end{rem}

We conclude the section with a technical lemma that we will
use to prove that a decorated concordance induces a non-trivial
map between the $E^\infty$ pages of the spectral sequences arising
from the knot filtrations.

\begin{lem}
\label{lem:review}
Let $f \colon C \to \bar C$ be a filtered map of degree zero
between non-acyclic bounded filtered complexes over $\FF_2$ such that
\begin{enumerate}
\item\label{it:F} {$H(C) \cong \mb F_2$ and $H(\bar C) \cong \mb F_2$,}
\item\label{it:tau} {$\tau(C) = \tau(\bar C)$, and}
\item{$H(f) \colon H(C) \to H(\bar C)$ is an isomorphism.}
\end{enumerate}
Then $E^\infty_\tau(C) \cong \mb F_2$, $E^\infty_\tau(\bar C) \cong \mb F_2$,
and the map $f^\infty \colon E^\infty_\tau(C) \to E^\infty_\tau(\bar C)$ is
also an isomorphism.
\end{lem}

\begin{proof}
Since \eqref{it:F} and \eqref{it:tau} hold,
by Theorem \ref{thm:spectralsequence} and Definition \ref{def:tauinvariant}, there are canonical isomorphisms
\[
E^\infty_\tau(C) \cong H(C) \cong \FF_2 \,\,\, \mbox{and} \,\,\, E^\infty_\tau(\bar C) \cong H(\bar C) \cong \FF_2.
\]
The commutativity of the following diagram concludes the proof.
\[
\begin{gathered}[b]
\xymatrix{
E^\infty_\tau(C) \ar[r]^{f^\infty} \ar[d]^-{\vsimeq} & E^\infty_\tau(\bar C) \ar[d]^-{\vsimeq} \\
H(C) \ar[r]^{\simeq}_{H(f)} & H(\bar C)}\\[-\dp\strutbox]
\end{gathered}
\qedhere
\]
\end{proof}

%----------------------------------------------------------------------------------

\section{Concordance maps preserve the knot filtration} \label{sec:filtration}

\subsection{The knot filtration}

Let $K$ be a null-homologous knot in a closed oriented 3-manifold $Y$.
Ozsv\'ath and Szab\'o~\cite{OSz3}, and independently Rasmussen~\cite{Ras},
proved that~$K$ gives rise to a filtration of the
Heegaard Floer chain complex~$\CFh(Y)$, well-defined up to filtered chain homotopy equivalence,
called the knot filtration. Such a filtration can be defined in terms of the Alexander grading;
see also \cite[Section~2.3]{OSz3}.

\begin{defn}
\label{def:Alexandergrading}
Let $S$ be a Seifert surface for the knot~$K$,
and let $(\Sigma, \bolda, \boldb, w, z)$ be a doubly pointed Heegaard
diagram for~$K$, as defined by Ozsv\'ath and Szab\'o~\cite{OSz3}.
Given a generator~$\x \in \T_\a \cap \T_\b$,
its $S$-Alexander grading is
\[
\A_S(\x) = \frac{1}{2} \langle c_1(\underline{\s}(\x)), [S]\rangle,
\]
where $\underline{\s}(\x)$ is the $\spinc$ structure on~$Y_0(K)$
extending~$\s(\x) \in \spinc(Y)$. We denote the corresponding filtration by~$\F_S$.
\end{defn}

\begin{rem} \label{rem:Alexandergrading}
Consider the sutured manifold $Y(K) = (M,\g)$ complementary to~$K$.
As in the proof of~\cite[Theorem~1.5]{decomposition} on page~27,
let~$t$ be the trivialization of~$v_0^\perp$ given by a vector field tangent
to~$\partial M$ in the meridional direction. Then
\[
\A_S(\x) = \frac{1}{2} \langle c_1(\rs(\x),t), [S]\rangle,
\]
where $\rs(\x)$ now denotes an element of $\spinc(M,\g)$.

If $Y$ is a rational homology 3-sphere,
all Seifert surfaces of $K$ are homologous in the knot
exterior, so the Alexander grading does not depend on~$S$,
and we simply denote it by $\A(\x)$, and the filtration by~$\F(\x)$.
\end{rem}

The following lemma describes how the relative Alexander grading can be read off the Heegaard
diagram; cf.~\cite[Lemma~2.5]{OSz3} and \cite[page~25]{Ras}.

\begin{lem}
\label{lem:Alexanderdomain}
Let $(\Sigma, \bolda, \boldb, w, z)$ be a Heegaard
diagram for a null-homologous knot $K$ in a 3-manifold~$Y$, and
let~$S$ be a Seifert surface for~$K$. If $\phi \in \pi_2(\x,\y)$, then
\[
n_z(\phi) - n_w(\phi) = \A_S(\x) - \A_S(\y).
\]
\end{lem}

\subsection{Knot filtration and concordances}

Our aim is to prove that the knot filtration is preserved by the
chain maps induced by concordances.

\begin{thm}
\label{thm:main1}
Let $\CC$ be a decorated concordance from $(K_0, P_0)$ to
$(K_1, P_1)$, and let $(\S_i,\bolda_i,\boldb_i,w_i,z_i)$
be a doubly pointed diagram representing~$(K_i, P_i)$ for $i \in \{0,1\}$.
Then there is a chain map
\[
f_{\CC} \colon \CFh(\S_0,\bolda_0,\boldb_0,w_0,z_0) \to \CFh(\S_1,\bolda_1,\boldb_1,w_1,z_1)
\]
preserving the knot filtration; i.e., for every
generator $\x \in \T_{\a_0} \cap \T_{\b_0}$,
\[
\A( f_{\CC}(\x) ) \leq \A(\x),
\]
such that $f_\CC$ induces the identity of $\HFh(S^3)$ on the total homology, and~$F_\CC$
on the homology of the associated graded complexes.
\end{thm}

%A straightforward consequence of Theorem~\ref{thm:main1} is the following.
%
%\begin{corollary}
%Let $\CC$ be a decorated concordance from $K_0 \subseteq S^3$ to
%$K_1 \subseteq S^3$. Then the map $F_{\mc C}$ induced on the
%knot Floer homology preserves the Alexander grading, that is for every
%homogeneous element $\x \in \HFK(K_0)$
%\[\A( F_{\mc C} (\x) ) = \A(\x).\]
%\end{corollary}

Theorem \ref{thm:main1} yields a
morphism of spectral sequences in the sense of Definition~\ref{def:morphism},
hence we have the following corollary.

\begin{thm} \label{thm:spectral}
Suppose that $\CC$ is a decorated concordance from $(K_0, P_0)$ to
$(K_1, P_1)$. Then there is a morphism of spectral sequences
from $\HFKh(K_0,P_0) \Rightarrow \HFh(S^3)$ to
$\HFKh(K_1,P_1) \Rightarrow \HFh(S^3)$ such that the map induced on
the $E^1$ page is~$F_C$, and the map induced on the $E^\infty$~page
is~$\Id_{\HFh(S^3)}$.
\end{thm}

\begin{proof}
Suppose that $\CC= (X,F,\sigma)$.
Since $H_1(X) = H_2(X) = 0$, it follows from the work of Ozsv\'ath and Szab\'o
\cite[Theorem 1.1]{OSz14} that $\tau(K_0) = \tau(K_1)$.
Indeed, the knot $K = K_0 \# \overline{K_1}$ bounds a disk in a homology 4-ball $W$
with boundary $S^3$, and hence $\tau(K) = \tau(K_0) - \tau(K_1) = 0$ by \cite[Theorem 1.1]{OSz14}.
By Theorem~\ref{thm:main1}, we have a filtered map $f_\CC$ that induces an
isomorphism on the total homology. We can therefore apply Lemma~\ref{lem:review}
to conclude that the map induced on the $E^\infty$ page is also an isomorphism.
\end{proof}

\begin{defn}
We say that an element $x \in \HFKh(K,P)$ \emph{survives} the spectral sequence to~$\HFh(S^3) \cong \Z_2$
if there is a sequence of cycles $x_i \in E^i$ for $i \ge 1$ such that $x_1 = x$ and $x_{i+1} = [x_i]$;
we denote the set of such elements by~$A(K)$.
Furthermore, we have a partition $A(K) = A_0(K) \cup A_1(K)$,
where $A_j(K)$ consists of those elements for which $x_i = j \in \Z_2$
for $i$ sufficiently large (note that the spectral
sequence is bounded).
\end{defn}

The subset $A_0(K)$ is a linear subspace of $A(K)$, and $A_1(K)$ is an
affine translate of $A_0(K)$. Each set $A(K)$, $A_0(K)$, and $A_1(K)$
is a knot invariant.

It follows from the definition of the Ozsv\'ath-Szab\'o $\tau$-invariant~\cite{OSz14} that
\begin{equation}
\label{eqn:Aandtaubasic}
A_1(K) \cap \HFKh(K, i)
\,\,\,
\begin{cases}
= \emptyset & \mbox{if $i \not= \tau(K)$} \\
\not= \emptyset & \mbox{if $i = \tau(K)$.}
\end{cases}
\end{equation}
If $a \in A_1(K)$, let $a_0$ denote the homogeneous component
of $a$ in homological grading zero. It is straightforward to check that $a_0$
survives the spectral sequence.
Since the homological grading on $\CFKh$ is inherited from the one on $\CFh$,
and since the homological grading of $1 \in \HFh(S^3)$ is zero, it follows that
$a_0 \in A_1(K)$. Combined with Equation~\eqref{eqn:Aandtaubasic}, this
implies that
\begin{equation}
\label{eqn:Aandtau}
A_1'(K) := A_1(K) \cap \HFKh_0(K, \tau(K)) \not = \emptyset.
\end{equation}
Notice that $A_1'(K)$ is also a knot invariant.

The following result is a straightforward consequence of
Theorem~\ref{thm:spectral}, Proposition~\ref{prop:homogeneous},
and Equation~\eqref{eqn:Aandtau}, and implies Corollary~\ref{cor:non-vanishing}
of the introduction.

\begin{cor} \label{cor:nonvanishing}
Let $\CC = (X,F,\sigma)$ be a decorated concordance from $(K_0, P_0)$ to
$(K_1, P_1)$, and let $\tau = \tau(K_0) = \tau(K_1)$.
Then, for $j \in \{0,1\}$,
\[
F_\CC(A_j(K_0))  \subseteq A_j(K_1)
\]
and hence it is non-zero from $\HFKh_0(K_0, P_0,\tau)$ to $\HFKh_0(K_1,P_1,\tau)$.
\end{cor}

\begin{proof}
The fact that $F_\CC(A_j(K_0)) \subseteq A_j(K_1)$ follows from Theorem~\ref{thm:spectral}.
In Section~\ref{sec:homol}, we shall see that $F_\CC$ preserves the homological grading.
Then, by Proposition~\ref{prop:homogeneous}, $F_\CC$ maps $\HFKh_0(K_0, P_0, \tau)$ to
$\HFKh_0(K_1, P_1, \tau)$. So we only need to prove that this map is non-zero.

By Equation~\eqref{eqn:Aandtau}, we have $A_1'(K_0) \neq \emptyset$; let $x \in A_1'(K_0)$.
Then, by the previous paragraph,
\[
F_\CC(x) \in A_1(K_1) \cap \HFKh_0(K_1,\tau) = A'_1(K_1),
\]
hence $F_\CC(x) \neq 0$.
\end{proof}

%\textbf{APPLY ARGUMENT FOR MIRROR OF K RESTRICTED TO $A(K)$ TO GET CANONICAL ELEMENT.}

We now turn to the proof of Theorem~\ref{thm:main1}, which will take
the rest of this section.

\subsection{Triviality of the gluing map}
\label{sec:gluingmap}

Given a sutured manifold cobordism $\W = (W, Z, [\xi])$
from $(M_0,\g_0)$ to~$(M_1,\g_1)$, the map
\[
F_\W \colon \SFH(M_0,\g_0) \to \SFH(M_1,\g_1)
\]
is the composition $F_{\Ws} \circ \Phi_{-\xi}$,
where
\[
\Phi_{-\xi} \colon \SFH(M_0,\g_0) \to \SFH(N,\g_1)
\]
is the gluing map of Honda, Kazez, and Mati\'c~\cite{TQFT}
for the sutured submanifold $(-M_0,-\g_0)$
of $(-N, -\g_1)$ with $N = M_0 \cup (-Z)$,
and $F_{\Ws}$ is a ``surgery map'' corresponding to handles
attached along the \emph{interior} of the sutured manifold~$N$.
The cobordism~$\Ws$ is a \emph{special cobordism}, meaning
its vertical part is a product and the contact structure on it is
$I$-invariant.

If $\CC=(X,F, \sigma)$ is a decorated concordance from~$(K_0,P_0)$ to~$(K_1,P_1)$,
let $\W = \W(\CC)$ be the complementary sutured manifold cobordism from
$S^3(K_0,P_0) = (M_0,\g_0)$ to~$S^3(K_1,P_1) = (M_1,\g_1)$.
Let $T^2 \times I$ be a collar neighborhood of~$\partial M_0$ such that
$T^2 \times \{1\}$ is identified with~$\partial M_0$.
Since the dividing set on~$F$ consists of two arcs connecting
the two components of~$\partial F$,
there is a diffeomorphism $d \colon T^2 \times I \to Z$
such that~$\xi' = d^*(\xi)$ is an $I$-invariant contact structure on~$T^2 \times I$,
and hence induces the trivial gluing map by \cite[Theorem~6.1]{TQFT}.
More precisely, if we write $M_0' = \bar{M_0 \setminus (T^2 \times I)}$ and~$\g_0'$
for the projection of~$\g_0$ to~$T^2 \times \{0\}$,
then there is a diffeomorphism $\varphi \colon (M_0',\g_0') \to (M_0,\g_0)$
supported in a neighborhood of $T^2 \times \{0\}$ such that
\[
\Phi_{-\xi'} = \varphi_* \colon \SFH(M_0',\g_0') \to \SFH(M_0,\g_0).
\]
Let $D \colon M_0 \to N$ be the diffeomorphism that agrees with
$\varphi$ on~$M_0'$ and with~$d$ on~$T^2 \times I$, smoothed along~$T^2 \times \{0\}$.
By the diffeomorphism invariance of the gluing construction, the diagram
\[
\xymatrix{
\SFH(M_0',\g_0') \ar[r]^{\varphi_*} \ar[d]^{\Phi_{-\xi'}} &\SFH(M_0,\g_0) \ar[d]^{\Phi_{-\xi}} \\
\SFH(M_0,\g_0) \ar[r]^{D_*} &\SFH(N,\g),
}
\]
is commutative, hence $\Phi_{-\xi} = D_*$.

We now show that $D_*$ preserves the Alexander grading on the chain level.
If we glue $D^2 \times S^1$ to~$N$ along~$\partial N$ such that the meridian is glued to
a suture in~$s(\g_1)$, we obtain a 3-manifold $Y$ diffeomorphic to~$S^3$,
and the image of~$\{0\} \times S^1$ is a knot~$K'$ in~$Y$. We can canonically
extend~$D$ to a diffeomorphism from $(S^3,K_0)$ to~$(Y,K')$.
Given a knot diagram $\H_0 = (\S_0,\bolda_0,\boldb_0,w_0,z_0)$ for $(S^3,K_0)$,
its image $D(\H_0)$ is a diagram of $(Y,K')$. Given a Seifert surface~$S$
of~$K_0$ and a generator $\x \in \T_{\a_0} \cap \T_{\b_0}$,
the image $D(S)$ is a Seifert surface of~$K'$, and $D(\x)$ satisfies
\[
\langle\, c_1(\rs(\x),t), [S] \,\rangle = \langle\, c_1(\rs(D(\x)), D_*(t)), [D(S)] \,\rangle.
\]
As~$D(\g_0) = \g_1$, the trivialization $D_*(t)$ points in the meridional direction for~$K'$,
and it follows that $\A(\x) = \A(D(\x))$.
It is apparent from the above discussion that we can identify $(S^3,K_0)$ and $(Y,K')$
via~$D$, so from now on we will think of~$\W$ as a special cobordism from~$(S^3,K_0)$
to~$(S^3,K_1)$.

\subsection{Notation}
\label{sec:notation}

\begin{figure}[t]
\begin{center}
\includegraphics{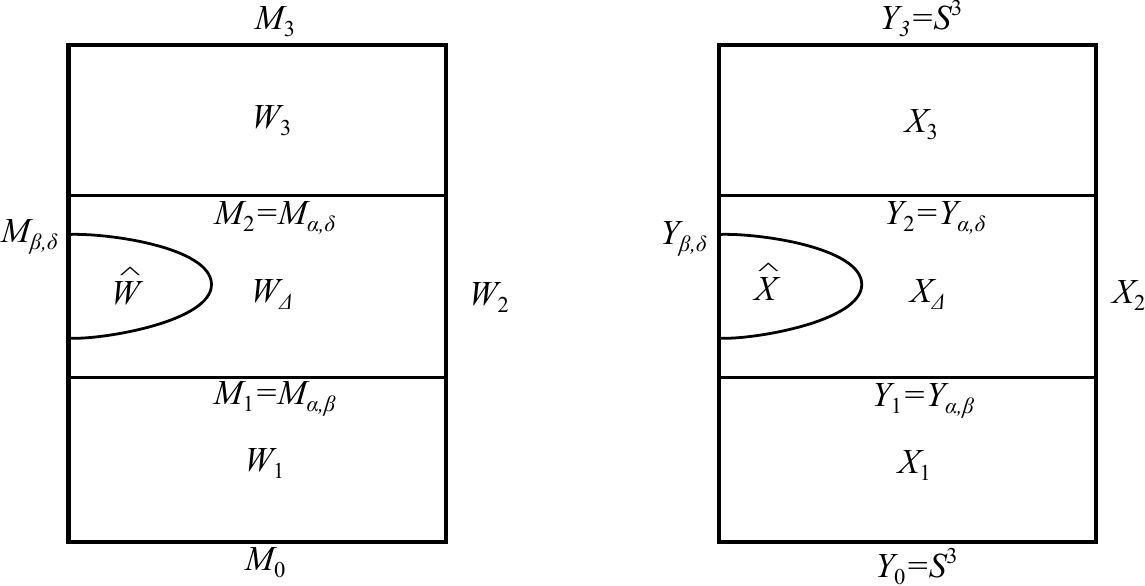}
\caption{The left-hand side shows the sutured cobordism
$\W=(W, Z, [\xi])$, and how we split it into different pieces.
%Notice that $\g$ denotes the sutures for all sutured manifolds and
%$Z$ and $[\xi]$ denote respectively the vertical part of the boundary
%and the equivalence class of contact structure for all sutured cobordisms.
The picture on the right-hand side shows the cobordism of 3-manifolds~$X$,
and the corresponding decomposition into smaller cobordisms.}
\label{fig:notation}
\end{center}
\end{figure}

In this subsection, we fix the notation for the rest of the paper.
Recall that $(K_0, P_0)$ and $(K_1, P_1)$ denote two decorated knots
in $S^3$, and that we have a decorated concordance $\CC = (X,F,\sigma)$
from $(K_0, P_0)$ to $(K_1, P_1)$.

We denote by $\W = (W, Z, [\xi])$ the sutured cobordism $\W(\CC)$ associated
to the knot concordance $\CC$. It follows from the discussion in
Section~\ref{sec:gluingmap} that $\W$ can be thought of as a special
cobordism. The $4$-manifold $W$ can be obtained by attaching
to~$M_0 \times I$ along the interior of $M_0 \times \{1\}$ a sequence of
4-dimensional 1-handles, followed by 2-handles, and finally 3-handles.
We denote the number of $i$-handles by $c_i$ for $i \in \{1,2,3\}$,
and often write $p$ for~$c_1$ and $\ell$ for~$c_2$.
We split the cobordism $\W$ into three parts $\W_1$,
$\W_2$, and $\W_3$, in such a way that $\W_i = (W_i, Z_i, [\xi_i])$ is a cobordism
from $(M_{i-1}, \gamma_{i-1})$ to $(M_i, \gamma_i)$,
and is the trace of the $i$-handle attachments;
see the left-hand side of Figure~\ref{fig:notation}.
%Each cobordism $\W_i$ is a triple $(W_i, Z, [\xi])$,
%where -- with slight abuse of notation -- we denote by $Z$ and $\xi$
%the vertical boundary and the contact structure for all cobordisms $\W_i$.
%Similarly, we denote by $\gamma$ the sutures of all
%the sutured manifolds $(M_i, \gamma)$.
Notice that, by construction, $(M_0, \g_0) = S^3(K_0, P_0)$ and $(M_3, \g_3) = S^3(K_1, P_1)$.

In order to represent sutured manifolds, we use Heegaard diagrams
with basepoints. If $w$, $z \in \S \setminus (\bolda \cup \boldb)$, the Heegaard diagram
$\H=(\S, \bolda, \boldb, w,z)$ represents the complement of a knot
in a 3-manifold. In order to recover the sutured Heegaard diagram as originally defined
by the first author~\cite{sutured}, one should remove a small disk around each basepoint.

Let $\cT=(\S, \bolda, \boldb, \boldd,w,z)$
be a doubly-pointed triple diagram for the cobordism~$\W_2$ (see
Section~\ref{sec:triplyperiodic}), where $d = |\bolda| = |\boldb| = |\boldd|$.
Furthermore, suppose that the 2-handles are  attached along an $\ell$-component framed link~$\L$.
Then we further split the manifold $W_2$ into two pieces according to \cite[Proposition~6.6]{cob}:
The piece $\W_{\a,\b,\d} = (W_\triangle, Z_\triangle, \xi_\triangle)$
denotes the sutured manifold cobordism obtained from the triangle
construction in \cite[Sections~5 and~6]{cob}, while $\W_\b(\L) = (\hat W, \hat Z, \hat \xi)$
is a sutured manifold cobordism from
\[
(R_+(\g_1),\partial R_+(\g_1) \times I)\# \left(\#_{i=1}^{d - \ell} (S^2 \times S^1) \right)
\]
to $\emptyset$.
The horizontal boundary of $\hat W$ is the sutured manifold $M_{\b, \d}$,
defined by the diagram $(\S, \boldb,\boldd, w,z)$.
By analogy, we also use the notation $M_{\a,\b} \cong (M_1, \gamma_1)$
and $M_{\a, \d} \cong (M_2, \g_2)$.

We can fill in the vertical boundary of the sutured cobordism~$\W$ by gluing
$D^2 \times S^1 \times I$ along $S^1 \times S^1 \times I$ to~$Z$
such that $S^1 \times \{(1,0)\}$ is glued to a meridian of $K_0$ to obtain
cobordisms of closed 3-manifolds rather than knot complements. In terms
of Heegaard diagrams, this amounts to forgetting the $z$ basepoints.
We denote the closed $3$-manifolds by the letter $Y$ rather than $M$.
As for the cobordisms, we use the letter $X$ instead of the letter $W$.
See the right-hand side of Figure~\ref{fig:notation}.

Lastly, let $S_0 \subseteq M_0$ and $S_3 \subseteq M_3$ be Seifert
surfaces for $K_0$ and $K_1$, respectively. Since $(M_1, \gamma_1)$ is
obtained from $(M_0,\gamma_0)$ by taking connected sums with copies
of $S^1 \times S^2$, the surface $S_0$ also defines a surface $S_1 \subseteq M_1$,
which is contained in the $M_0$ summand of $M_1$.
Analogously, the Seifert surface $S_3$ induces a Seifert surface $S_2 \subseteq M_2$.

\subsection{Definition of the chain map~$f_\CC$}

We now define the chain map~$f_\CC$.
Given an admissible doubly-pointed diagram~$\H = (\S,\bolda,\boldb,w,z)$ for a decorated knot $(Y,K,P)$,
we denote by $\CFh(\H)$ the Heegaard Floer chain complex that counts disks avoiding~$w$ and filtered by~$z$.
Its homology is $\HFh(Y,w)$, while the homology of the associated graded complex $\CFKh(\H)$ is $\HFKh(Y,K,P)$.

Suppose that the 1-handles are attached along $p$ framed pairs of points~$\bP \subset M_0$.
Pick an admissible diagram~$\H^0$ of~$(M_0,\g)$ subordinate to~$\bP$, and let
\[
f_{\H^0,\bP} \colon \CFh(\H^0) \to \CFh(\H^0_{\bP})
\]
be the 1-handle map defined in~\cite[Definition~7.5]{cob}.
The 2-handles are attached along an $\ell$-component framed link~$\L \subset M_1$. Choose
an admissible diagram~$\H^1$ subordinate to $\L$, and let
\[
f_{\H^1,\L} \colon \CFh(\H^1) \to \CFh(\H^1_\L)
\]
be the 2-handle map defined in \cite[Definition~6.8]{cob}, on the chain level.
This map counts triangles that avoid~$w$ but might pass through ~$z$.
Finally, let~$\H^2$ be an admissible diagram of~$(M_2,\g)$ subordinate to framed spheres~$\bS \subset M_2$
corresponding to the 3-handles. The corresponding 3-handle map
\[
f_{\H^2,\bS} \colon \CFh(\H^2) \to \CFh(\H^2_\bS)
\]
was introduced in \cite[Definition~7.8]{cob}.

Given admissible diagrams~$\H$ and~$\H'$ of a sutured manifold~$(M,\g)$, we refer the reader to~\cite[Section~5.2]{cob}
for the definition of the canonical isomorphism
\[
F_{\H,\H'} \colon \SFH(\H) \to \SFH(\H').
\]
We can obtain a chain level representative by connecting~$\H$ and~$\H'$
through a sequence of ambient isotopies, (de)stabilizations, and equivalences of the
attaching sets. If~$(M,\g)$ is complementary to a knot~$(Y,K)$, we can view this as a sequence of
moves on knot diagrams. Each induces a chain homotopy equivalence on~$\CFh$ preserving
the knot filtration according to~\cite{OSz3, Ras}, and induces an isomorphism both on the
homology of the whole complex (isomorphic to~$\HFh(Y)$),
and the homology of the associated graded complex (isomorphic to~$\HFKh(Y,K)$).
Note that the triangle maps corresponding to changing the attaching curves
do not pass over~$w$ but might cross~$z$, so they are in fact naturality
maps for the closed 3-manifold and \emph{not} the knot.
We proved in~\cite{naturality} that the maps on the homology are independent of the sequence of
moves connecting~$\H$ and~$\H'$. We write~$f_{\H,\H'}$ for the chain level representative
of~$F_{\H,\H'}$ described above.
With the above notation in place, we set
\[
f_{\CC} := f_{\H^2,\bS} \circ f_{\H^1_\L, \H^2} \circ f_{\H^1,\L} \circ f_{\H^0_{\bP},\H^1} \circ f_{\H^0, \bP},
\]
from $\CFh(\H^0)$ to $\CFh(\H^2_{\bS})$. Note that each of the diagrams involved in the above formula
can be viewed as a knot diagram after gluing disks along~$s(\g)$ that do not change during the
cobordism, so we can distinguish~$z$ and~$w$ throughout. If we are given
diagrams $\H$ of $(M_0,\g_0)$ and $\H'$ of $(M_3,\g_3)$, then we have to pre- and post-compose
the above map~$f_\CC$ with $f_{\H^2_\bS,\H'}$ and $f_{\H,\H^0}$.

We split the proof of Theorem~\ref{thm:main1} into a number of steps,
and we prove that for each $\W_i$ the knot filtration is preserved.
%This implies that~$f_\CC$ preserves the knot filtration, since
%we already know that the naturality maps do.

\subsection{1- and 3-handles}

First, consider the case of the 1-handle attachments along the framed pairs of points~$\bP \subset \text{Int}(M_0)$.
As in Section~\ref{sec:notation}, we write $\W_1 := \W(\bP)$ for the trace of the surgery along~$\bP$;
this is a cobordism from~$(M_0,\g_0)$ to~$(M_1,\g_1)$.
Recall~\cite[Section~7]{cob} that there is an isomorphism
$\spinc(\W_1) \cong \spinc(M_0, \gamma_0)$.
Furthermore, a $\spinc$ structure $\rs \in \spinc(M_1,\g_1)$ extends to~$\W_1$
if and only if $c_1(\rs)$ vanishes on the belt spheres of all the 1-handles.
Given~$\rs \in \spinc(\W)$, we write $\rs_0$ for its restriction to $(M_0,\g_0)$,
and $\rs_1$ for its restriction to $(M_1,\g_1)$.

\begin{lem}
\label{lem:1handles}
Let $\rs_0 \in \spinc(M_0,\gamma_0)$, and let
$\rs_1 \in \spinc(M_1,\gamma_1)$ denote the corresponding $\spinc$ structure.
Then
\[
\langle c_1(\rs_0, t) , [S_0] \rangle = \langle c_1(\rs_1, t), [S_1]\rangle.
\]
\end{lem}

\begin{proof}
This is a consequence of the naturality of the first Chern class and
the fact that both $S_0$ and $S_1$ are actually contained in $M_0 \setminus N(\bP)$.
We can suppose that $S_0$ is properly embedded in
$M_0 \setminus N(\bP)$. By definition,
$S_1$ is a surface contained in $M_0 \setminus N(\bP) \subseteq M_1$
that is isotopic to $S_0$ in $M_0 \setminus N(\bP)$.

Since $S_0$ and $S_1$ are isotopic in $M_0 \setminus N(\bP)$
and $\rs_1 |_{M_0 \setminus N(\bP)} = \rs_0 |_{M_0 \setminus N(\bP)}$,
by the naturality of the first Chern class
\[
\begin{split}
\langle c_1(\rs_1, t) , [S_1] \rangle
&=\langle c_1(\rs_1|_{M_0 \setminus N(\bP)}, t) , [S_1] \rangle \\
&=\langle c_1(\rs_0|_{M_0 \setminus N(\bP)}, t) , [S_0] \rangle \\
&=\langle c_1(\rs_0, t), [S_0]\rangle.
\end{split}
\]
Notice that the trivialization $t$ of the vector field $v_0$ on $\de M_0 = \de M_1$
does not change because the boundary is left unaffected by the surgery.
\end{proof}

\begin{rem}~\label{rem:1handles}
Since $c_1(\rs_1,t)$ vanishes on the belt spheres of the 1-handles,
the above result also holds for an arbitrary Seifert surface~$S_1$.
\end{rem}

\begin{cor}
\label{cor:1handles}
The map $f_{\H^0, \bP}: \CFh(\H^0)  \to \CFh(\H^0_\bP)$ preserves the Alexander
grading (cf.~Definition~\ref{def:Alexandergrading}) with respect to arbitrary
Seifert surfaces $S_0$ and $S_1$; i.e.,
\[
\A_{S_1}(f_{\H^0, \bP}(\x)) = \A_{S_0}(\x)
\]
for any $\x \in \T_{\a^0} \cap \T_{\b^0}$, where $\H^0 = (\S^0,\bolda^0,\boldb^0,w^0,z^0)$.
\end{cor}

\begin{proof}
This is a straightforward consequence of Lemma~\ref{lem:1handles}, Remark~\ref{rem:1handles},
and the fact that the relative $\spinc$ structure induced by $\rs(\x)$
on $(M_1, \gamma)$ is exactly $\rs(f_{\H^0, \bP}(\x))$.
\end{proof}

A dual reasoning gives the following results for the map $f_{\H^2, \bS}$,
which are analogous to Lemma~\ref{lem:1handles} and Corollary~\ref{cor:1handles}.

\begin{lem}
\label{lem:3handles}
Let $\rs_3 \in \spinc(M_3,\gamma_3)$, and let
$\rs_2 \in \spinc(M_2,\gamma_2)$ denote the corresponding $\spinc$ structure.
Then
\[
\langle c_1(\rs_2, t) , [S_2] \rangle = \langle c_1(\rs_3, t), [S_3]\rangle.
\]
\end{lem}

\begin{cor}
\label{cor:3handles}
The map $f_{\H^2, \bS}: \CFh(\H^2)  \to \CFh(\H^2_\bS)$ preserves the Alexander
grading with respect to arbitrary Seifert surfaces $S_2$ and $S_3$; i.e.,
\[
\A_{S_3}(f_{\H^2, \bS}(\x)) = \A_{S_2}(\x)
\]
for any $\x \in \T_{\a^2} \cap \T_{\b^2}$ such that $f_{\H^2, \bS}(\x) \neq 0$,
where $\H^2 = (\S^2,\bolda^2,\boldb^2,w^2,z^2)$.
\end{cor}

\subsection{2-handles}

The proof that the Alexander grading is preserved under the attachment of the
2-handles is less straightforward than in the case of 1-handles and 3-handles.

\begin{lem}
\label{lem:2handles1}
Let $\CC$ be a decorated concordance from $(K_0,P_0)$ to $(K_1,P_1)$.
With the notation of Section~\ref{sec:notation},
let $\W_2$ denote the $2$-handle cobordism from $(M_1, \gamma_1)$
to $(M_2, \gamma_2)$ obtained by surgery along a framed link~$\L$,
and let $S_1$ and $S_2$ be corresponding Seifert surfaces.
Then there is an admissible doubly pointed triple diagram $(\S,\bolda,\boldb,\boldd,w,z)$
subordinate to a bouquet for~$\L$ as follows: If $\x \in \T_\a \cap \T_\b$ is such
that $\s(\x) \in \spinc(Y_{\a, \b})$ extends to $X_1$, then
for any $\y \in \T_\a \cap \T_\d$ that appears with non-zero coefficient
in $f_{\H^1, \L}(\x)$, and such that $\s(\y) \in \spinc(Y_{\a, \d})$
extends to $X_3$, we have
\[
\F_{S_2}(\y) \le \F_{S_1}(\x).
\]
Moreover, if $\psi$ is a holomorphic triangle connecting $\x$, $\theta$
(the top-graded generator of $\CFh(\S,\boldb,\boldd,w,z)$),
and $\y$ that does not cross $w$, then
\begin{equation}
\label{eq:shift}
\F_{S_2}(\y) = \F_{S_1}(\x) - n_z(\psi).
\end{equation}
\end{lem}

Notice that, in Lemma~\ref{lem:2handles1}, we consider
ordinary $\spinc$ structures rather than relative ones.
Recall that relative $\spinc$ structures are defined for sutured cobordisms,
which we denote by the letter $\W$, while ordinary $\spinc$ structures
are defined for cobordisms of 3-manifolds, which we denote by the letter
$X$, cf.~Figure~\ref{fig:notation}.

\begin{idea}
Consider an admissible Heegaard triple diagram $(\S,\bolda,\boldb,\boldd)$ subordinate to a bouquet
for a framed link $\L$, as explained in \cite[Section~6]{cob}.
Suppose that $\x \in \T_\a \cap \T_\b$ is such that $\s(\x) \in \spinc(Y_{\a, \b})$ extends to $X_1$.
Let $\theta \in \T_\b \cap \T_\d$ be the top-graded generator of $\CFh(\S,\boldb,\boldd)$,
and let $\y \in \T_\a \cap \T_\d$ be such that $\s(\y) \in \spinc(Y_{\a, \d})$ extends to $X_3$.
Given a holomorphic triangle $\psi \in \pi_2(\x,\theta,\y)$, let
\[
c = \A_{S_2}(\y) - \A_{S_1}(\x) + n_z(\psi) - n_w(\psi).
\]
First, we prove that $c$ is independent of $\psi$, $\x$, and $\y$.
If $\psi_1$, $\psi_2 \in \pi_2(\x,\theta,\y)$, then the domain
$\D(\psi_1)-\D(\psi_2)$ is triply periodic.
If we prove that, for every triply periodic domain $D$, we have
\[
n_z(D) - n_w(D) = 0,
\]
then $c$ is independent of $\psi$. For this reason, the next subsection is devoted
to the study of triply periodic domains in the setting of Lemma~\ref{lem:2handles1}.

Given different intersection points $\x' \in \T_\a \cap \T_\b$ and $\y' \in \T_\a \cap \T_\d$
such that $\s(\x') \in \spinc(Y_{\a, \b})$ extends to $X_1$ and
$\s(\y') \in \spinc(Y_{\a, \d})$ extends to $X_3$,
there are domains $D_\x$ connecting $\x$ with $\x'$ and $D_\y$ connecting
$\y$ with $\y'$ that do not pass through~$w$
(but might have non-trivial multiplicities at~$z$). Adding these domains to $D(\psi)$, we get
a triangle domain connecting $\x'$, $\theta$, and $\y'$ with the same~$c$ by Lemma~\ref{lem:Alexanderdomain}.

Then we show that $c = 0$ by isotoping~$\bolda$ to obtain a diagram where such
$\x$, $\y$, and $\psi$ as above exist, and invoke Lemma~\ref{lem:samespinc}.
Finally, if $\psi$ appears in the surgery map~$f_{\H^1, \L}(\x)$, then $n_w(\psi) = 0$
and it has a pseudo-holomorphic representative, so $n_z(\psi) \ge 0$.
Consequently, $A_{S_2}(\y) \le A_{S_1}(\x)$, as desired.
\qed
\end{idea}

We now explain the missing details in the above outline.

\subsection{Triply periodic domains}
\label{sec:triplyperiodic}

%The idea of using the triply periodic domains to study grading shifts
%comes from~\cite{manolescu2007khovanov}, where the authors found a nice
%set of generators of the set of periodic domains in order to analyze the
%grading shift induced by each periodic domain by looking only at the
%generators.

The following argument was motivated by the work of Manolescu and
Ozsv\'ath~\cite{manolescu2007khovanov}.

\begin{defn}
A \emph{doubly pointed triple Heegaard diagram} is a tuple
\[
\cT = (\S,\bolda, \boldb, \boldd, w, z),
\]
where $\S$ is a closed,
oriented surface, and there is an integer $d \ge 0$ such that the sets
$\bolda$, $\boldb$ and $\boldd$ all consist of $d$ pairwise disjoint
simple closed curves in $\S \setminus \left\{w,z\right\}$ that are linearly
independent in $H_1(\Sigma\setminus\left\{w,z\right\})$.

We denote by $Y_{\a,\b}$, $Y_{\a,\d}$, and
$Y_{\b,\d}$ the 3-manifolds represented by the Heegaard diagrams
$(\S, \bolda, \boldb)$, $(\S,\bolda,\boldd)$, and
$(\S,\boldb,\boldd)$, respectively.
\end{defn}

\begin{defn}
Let $\cT = (\S,\bolda, \boldb, \boldd, w, z)$ be a doubly pointed triple
Heegaard diagram.
Let $D_1, \ldots, D_l$ denote the closures of the components of
$\S \setminus(\bolda \cup \boldb \cup \boldd)$.
Then the set of \emph{domains} in $\cT$ is
\[
D(\cT) = \Z \langle\, D_1, \ldots, D_l \,\rangle.
\]
We denote by $n_z(\D)$ (respectively $n_w(\D)$) the multiplicity of a domain $\D \in D(\cT)$ in the
region $D_i$ that contains $z$ (respectively~$w$).

A \emph{triply periodic domain} is an element $\P \in D(\cT)$
such that $\partial \P$ is a $\Z$-linear combination of curves in $\bolda \cup \boldb \cup \boldd$.
We denote the set of triply periodic domains by $\Pi_{\bolda,\boldb,\boldd}$.

A \emph{doubly periodic domain} is an element $\P \in D(\cT)$
such that $\partial \P$ is a $\Z$-linear combination of curves in either $\bolda \cup \boldb$,
or in $\boldb \cup \boldd$, or in $\bolda \cup \boldd$.
We denote the set of the three types of doubly periodic domains by $\Pi_{\a,\b}$,
$\Pi_{\a,\d}$, and $\Pi_{\b,\d}$, respectively.
\end{defn}

The following result states that every triply periodic domain in the diagram
describing the surgery map for~$\W_2$ can be written as a sum of doubly periodic
domains.

\begin{prop}
\label{prop:triplyperiodic}
%Fix an arbitrary handle decomposition of the cobordism $W = S^3 \times I$
%from $S^3$ to $S^3$. Denote by $W_1$, $W_2$ and $W_3$ the parts of the cobordism
%that correspond to the attachments of the 1-handles, 2-handles and 3-handles, respectively.
Let $(\S, \bolda, \boldb, \boldd)$ denote a Heegaard
diagram associated to the cobordism $X_2$. Then
\[
\Pi_{\a,\b,\d} = \Pi_{\a,\b} + \Pi_{\a,\d} + \Pi_{\b,\d}.
\]
\end{prop}

Given a triple diagram associated to a surgery on an $\ell$-component link $\L$,
one can construct a 4-manifold $X_\triangle$ as in \cite[Section~2.2]{OSz10}
(see \cite[Section~5]{cob} for the analogous construction in the sutured setting).
%in a way similar to the one described in \cite[Section~5]{cob}.
In $\partial X_\triangle$ naturally sit $Y_{\a,\b}$, $Y_{\a,\d}$,
and~$Y_{\b,\d}$ that are the 3-manifolds defined by the Heegaard diagrams
$(\S, \bolda,\boldb)$, $(\S,\bolda,\boldd)$, and $(\S,\boldb,\boldd)$, respectively.
The cobordism~$X_2$ corresponding to the attachment of the 2-handles is
obtained by gluing the 4-manifold $\hat X = \natural_{i=1}^\ell(S^1 \times D^3)$ to $X_\triangle$
along $Y_{\b,\d} \cong \#_{i=1}^\ell(S^1 \times S^2)$.

\begin{lem}[{\cite[Propositions~2.15 and~8.3]{OSz}}] \label{lem:triplyperiodic1}
Given a pointed triple Heegaard diagram $(\S, \bolda, \boldb, \boldd, z)$,
there are isomorphisms
\begin{itemize}
	\item $\pi_{\a,\b} \colon \Pi_{\a,\b} \xrightarrow{\simeq} \Z \oplus H_2(Y_{\a,\b})$ and
	\item $\pi_{\a,\b,\d} \colon \Pi_{\a,\b,\d} \xrightarrow{\simeq} \Z \oplus H_2(X_\triangle)$.
\end{itemize}
In both cases, the projection onto the $\Z$ summand is given by $n_z$.
\end{lem}

%{\color{red} Explain how a 2-homology class is obtained from a periodic domain.
%Cite \cite[p.~1043]{OSz8} for the construction of the class in $\H_2(X)$.}

\begin{lem}
\label{lem:triplyperiodic2}
Given a pointed triple Heegaard diagram $(\S, \bolda, \boldb, \boldd, z)$,
the isomorphisms from Lemma~\ref{lem:triplyperiodic1} fit into the commutative diagram
\[
\xymatrix{
\Pi_{\a,\b} \ar[r]^-{\pi_{\a,\b}} \ar[d] & \Z \oplus H_2(Y_{\a,\b}) \ar[d]^{\Id_\Z \oplus i_*} \\
\Pi_{\a,\b,\d} \ar[r]^-{\pi_{\a,\b,\d}} & \Z \oplus H_2(X_\triangle),
}
\]
where $i \colon Y_{\a,\b} \to X_\triangle$ is the embedding.
\end{lem}

\begin{proof}
Let $\P$ be a doubly periodic domain in $\Pi_{\a,\b}$. By construction, the 2-chain in~$X_\Delta$
associated to $\P$ -- thought of as a triply periodic domain -- is
homotopic, hence homologous to $i_*(H(\P))$, where $H(\P)$ is the 2-chain in $Y_{\a,\b}$
obtained by capping off the boundary of the doubly periodic domain $\P$.
Therefore, the projections onto the second summand commute.
The projections onto the $\Z$ summands commute because in both cases they are obtained by
taking $n_z$.
\end{proof}

\begin{proof}[{Proof of Proposition~\ref{prop:triplyperiodic}}]
By Lemmas~\ref{lem:triplyperiodic1} and~\ref{lem:triplyperiodic2}, it is sufficient
to prove that the map
\[
\chi: H_2(Y_{\a,\b}) \oplus H_2(Y_{\a,\d}) \oplus H_2(Y_{\b,\d}) \longrightarrow H_2(X_\triangle)
\]
is surjective.

From the long exact sequence associated to the pair
$(X_\triangle, Y_{\a,\b} \sqcup Y_{\a,\d} \sqcup Y_{\b,\d})$, we see that the map $\chi$ is
surjective if and only if
\[
\varphi: H_2(X_\triangle, Y_{\a,\b} \sqcup Y_{\a,\d} \sqcup Y_{\b,\d}) \longrightarrow
H_1(Y_{\a,\b} \sqcup Y_{\a,\d} \sqcup Y_{\b,\d})
\]
is injective.
From the inclusion of pairs
\[
i_{\a,\b,\d} \colon (X_\triangle, Y_{\a,\b} \sqcup Y_{\a,\d} \sqcup Y_{\b,\d})
\hookrightarrow (X, X_1 \sqcup X_3 \sqcup \hat X),
\]
we obtain the commutativity of the following diagram:

\[
\xymatrix{
H_2(X_\triangle, Y_{\a,\b} \sqcup Y_{\a,\d} \sqcup Y_{\b,\d}) \ar[r]^-{\varphi} \ar[d]_-{(i_{\a,\b,\d})_*}^-{\vsimeq}
&H_1(Y_{\a,\b})  \oplus  H_1(Y_{\a,\d}) \oplus H_1(Y_{\b,\d}) \ar[d]^{(i_{\a,\b})_* \oplus (i_{\a,\d})_* \oplus (i_{\b,\d})_*} \\
H_2(X,X_1 \sqcup X_3 \sqcup \widehat{X}) \ar[r]^-{\simeq}_-{\tilde\varphi} & H_1(X_1) \oplus H_1(X_3) \oplus H_1(\widehat{X}),
}
\]
where $i_{\a,\b}$, $i_{\a,\d}$, and $i_{\b,\d}$ are the restrictions of $i_{\a,\b,\d}$
to $Y_{\a,\b}$, $Y_{\a,\d}$, and $Y_{\b,\d}$, respectively.
The map $(i_{\a,\b,\d})_*$ is an isomorphism by excision. The fact that
$\tilde\varphi$ is an isomorphism follows from the long exact sequence in
homology associated with the pair $(X, X_1 \sqcup X_3 \sqcup \hat X)$,
together with the fact that $H_2(X) = H_1(X) = 0$.

The commutativity of the above diagram implies that the map $\varphi$ is
injective, and therefore concludes the proof of the proposition.
\end{proof}

\begin{rem}
\label{rem:aftertp}
The important condition in Proposition~\ref{prop:triplyperiodic} is that the map
\[
\rho \colon H_2(X_1) \oplus H_2(X_3) \oplus H_2(\hat X) \To H_2(X)
\]
is surjective, which is obviously true as $H_2(X)=0$.
The surjectivity of $\rho$ is equivalent to the injectivity of $\tilde\varphi$,
which implies the injectivity of $\varphi$.
\end{rem}

In Proposition~\ref{prop:triplyperiodic}, we saw that, in the case of
a triple diagram describing the 2-handle attachments in the
cobordism $X$, every triply periodic domain can be expressed
as a sum of doubly periodic domains.
We now analyze the doubly periodic domains.

\begin{prop}
\label{prop:doublyperiodic}
Consider a null-homologous knot~$K$ in a 3-manifold~$Y$.
If $(\S, \bolda, \boldb,w,z)$ is a doubly pointed Heegaard diagram
for $(Y,K)$, then for every periodic domain $\P$,
\[
n_z(\P) - n_w(\P) = 0.
\]
\end{prop}

\begin{proof}
Let $H(\P) \in C_2(Y)$ be the 2-cycle obtained by capping off the
boundary of~$\P$ with the cores of the 3-dimensional 2-handles attached
to $\S \times I$ along $\bolda \times \{0\}$ and $\boldb \times \{1\}$.
Then $n_z(\P) - n_w(\P)$ is precisely the algebraic intersection number
of $H(\P)$ and $K$, which is zero as $K$ is null-homologous.
\end{proof}

\subsection{Representing homology classes}

Let $(\Sigma, \bolda, \boldb)$ be a Heegaard diagram for a 3-manifold
$Y$. It is straightforward to see that any element of $H_1(Y)$ can
be represented by a 1-cycle in $\Sigma$. In this subsection, we strengthen
this result for the case of concordances in the following sense.

\begin{lem}
\label{lem:H1Sigma}
Choose an arbitrary handle decomposition of the cobordism~$X$
from $S^3$ to $S^3$, and let $X_2$ denote the trace of the 2-handle attachments.
Suppose that $(\Sigma,\bolda,\boldb,\boldd,w,z)$ is a doubly pointed triple Heegaard
diagram subordinate to a bouquet for a link $\L$ that defines $X_2$.
Then the map
\[
i \colon H_1(\Sigma) \longrightarrow H_1(Y_{\a,\b}) \oplus H_1(Y_{\a,\d}),
\]
induced by the inclusions $\Sigma \hookrightarrow Y_{\a,\b}$ and
$\Sigma \hookrightarrow Y_{\a, \d}$, is surjective.
\end{lem}

In other words, given any two classes in the first
homologies of $Y_{\a, \b}$ and $Y_{\a,\d}$, there is a
1-cycle in $\Sigma$ that represents both simultaneously.

\begin{proof}
%Let $\bolde$ be those curves in $\boldb$ that are isotopic to components of~$\boldd$.
%Then we can write $\boldb = \bolde \sqcup \boldb'$, and $\boldd$ is $\bolde \sqcup \boldd'$, up to isotopy.
Consider the following short exact sequence of abelian groups:
\[
0 \rightarrow \frac{H_1(\Sigma)}{\spn{\bolda, \boldb} \cap \spn{\bolda,\boldd}}
\rightarrow \frac{H_1(\Sigma)}{\spn{\bolda, \boldb}} \oplus \frac {H_1(\Sigma)}{\spn{\bolda, \boldd}}
\rightarrow \frac{H_1(\Sigma)}{\spn{\bolda, \boldb, \boldd}}
\rightarrow 0.
\]
The middle term is isomorphic to $H_1(Y_{\bolda, \boldb}) \oplus H_1(Y_{\bolda,\boldd})$,
and the last term is isomorphic to $H_1(X_\triangle)$, where $X_\triangle$ is the
4-manifold obtained by the triangle construction, cf.~\cite[Proposition~8.2]{OSz}.
The short exact sequence above can then be rewritten as
\[
0 \rightarrow \frac{H_1(\Sigma)}{\spn{\bolda, \boldb} \cap \spn{\bolda, \boldd}}
\xrightarrow{f} H_1(Y_{\bolda, \boldb}) \oplus H_1(Y_{\bolda,\boldd})
\xrightarrow{g} H_1(X_\triangle)
\rightarrow 0.
\]

If we prove that $H_1(X_\triangle) = 0$, then by exactness we have that the map $f$ is surjective.
So the map $i$ in the statement of the lemma is surjective too,
because it is obtained by composing the following two maps:
\[
H_1(\Sigma)
\rightarrow \frac{H_1(\Sigma)}{\spn{\bolda, \boldb} \cap \spn{\bolda, \boldd}}
\xrightarrow{f} H_1(Y_{\a, \b}) \oplus H_1(Y_{\a,\d}).
\]
Therefore, we only need to prove that $H_1(X_\triangle) = 0$.
For this purpose, consider the Mayer-Vietoris long exact sequence
associated to the decomposition $X = A \cup B$, where
$A = X_\triangle$ and $B = X_1 \sqcup X_3 \sqcup \hat X$.
A portion of the long exact sequence is
\[
H_2(X)
\rightarrow H_1(A \cap B)
\xrightarrow{\iota} H_1(A) \oplus H_1(B)
\rightarrow H_1(X).
\]
Since $X$ has trivial first and second homology groups,
by exactness the map $\iota$ gives an isomorphism
\begin{equation}
\label{eq:hisomorphism}
%h \colon
H_1(Y_{\a,\b}) \oplus H_1(Y_{\a,\d}) \oplus H_1(Y_{\b,\d})
\xrightarrow{\sim} H_1(X_\triangle) \oplus H_1(X_1) \oplus H_1(X_3) \oplus H_1(\hat X).
\end{equation}

If $c_k$ denotes the number of $k$-handles in the decomposition
of the cobordism $X$ and $d = |\bolda|$, then it is straightforward to check that
\[
\begin{split}
H_1(Y_{\a, \b}) &\cong H_1(X_1) \cong \Z^{c_1}, \\
H_1(Y_{\b, \d}) &\cong H_1(\hat X) \cong \Z^{d-c_2}, \text{ and} \\
H_1(Y_{\a, \d}) &\cong H_1(X_3) \cong \Z^{c_3}.
\end{split}
\]
It now follows from Equation~\eqref{eq:hisomorphism} that $H_1(X_\triangle) = 0$,
which concludes the proof.
\end{proof}

\subsection{Proof of Lemma~\ref{lem:2handles1}}

%We are now ready to prove Lemma~\ref{lem:2handles1}.

%\begin{lem}
%Let $\CC$ be a concordance from $S^3(L_0)$ to $S^3(L_1)$.
%According to the notation explained in Section~\ref{sec:proofprel},
%let $\W_2$ denote the $2$-handle cobordism from $(M_1, \gamma)$
%to $(M_2, \gamma)$, and let $S_1$ and $S_2$ be the Seifert surfaces
%from Section~\ref{sec:proofprel}.
%Then, the $S_i$-Alexander filtration is preserved, that is
%\[
%\F_{S_2}(f_{\W_2}(\x)) \le \F_{S_1}(\x)
%\]
%where $f_{\W_2}$ is the map induced between the sutured Floer complexes.
%\end{lem}

%\begin{proof}[Proof of Lemma~\ref{lem:2handles1}]
The cobordism $\W_2$ can be represented via surgery on a framed $\ell$-component link $\L$.
Let $\cT = (\S, \bolda, \boldb, \boldd, w,z)$ be a doubly pointed triple
Heegaard diagram subordinate to a bouquet for the framed link $\L$.
As in \cite[Section~6]{cob}, we suppose that
$d= |\bolda| = |\boldb| = |\boldd|$, and that the curve~$\delta_i$
is an isotopic translates of~$\beta_i$ for $i \in \{\ell+1,\ldots,d\}$.

Following the notation established in Section~\ref{sec:notation} and in
Figure~\ref{fig:notation}, let $Y_{\a,\b}$, $Y_{\a,\d}$, and $Y_{\b,\d}$
denote the closed manifolds associated to
the Heegaard diagrams $(\S, \bolda, \boldb)$, $(\S, \bolda,\boldd)$,
and $(\S, \boldb, \boldd)$, respectively. Each of these closed manifolds contains a knot,
defined by the basepoints~$w$ and~$z$. We denote the knot exteriors -- thought of
as sutured manifolds -- by $M_{\a,\b}$, $M_{\a,\d}$, and
$M_{\b,\d}$. We let~$\g$ denote the sutures of all three sutured
manifolds.

Let $\s$ be the unique $\spinc$ structure on $X$.
By definition, $\s|_{X_\triangle}$ is the unique $\spinc$ structure on
$X_\triangle$ that extends to the whole cobordism $X$.
Suppose that $\x \in \T_\a \cap \T_\b$ and $\y \in \T_\a \cap \T_\d$
are such that $\s(\x) = \s|_{Y_{\a,\b}}$ and $\s(\y) = \s|_{Y_{\a,\d}}$.
Let $\theta \in \T_\b \cap \T_\d$ denote the top-graded generator.
Consider a Whitney triangle $\psi \in \pi_2(\x,\theta,\y)$,
possibly crossing the basepoints~$z$ and~$w$, and let
\begin{equation}
\label{eq:c}
c = \A_{S_2}(\y) - \A_{S_1}(\x) + n_z(\psi) - n_w(\psi).
\end{equation}

Our aim is to show that $c = 0$.
First, we show that $c$ is independent of the triangle~$\psi \in \pi_2(\x,\theta,\y)$
for fixed~$\x$ and~$\y$. Indeed, let $\psi_1$, $\psi_2 \in \pi_2(\x,\theta,\y)$. The domain
$\P = \D(\psi_1) - \D(\psi_2)$ is triply periodic.
By Proposition~\ref{prop:triplyperiodic}, $\P$ can be expressed as the sum of three doubly
periodic domains $\P_{\a,\b}$, $\P_{\b,\d}$, and $\P_{\a,\d}$.

%Since the Heegaard diagram $(\S, \bolda,\boldb,w,z)$ is subordinate to a bouquet for~$\L$,
%up to Heegaard moves, it is obtained by taking
%the connected sum of a doubly pointed Heegaard diagram for a knot in $S^3$ with
%copies of the standard admissible genus one Heegaard diagram for $S^1 \times S^2$. Then, by
%Proposition~\ref{prop:doublyperiodic}, we have that $n_z(\P_{\b,\d}) = n_w(\P_{\b,\d})$.

%The manifold $M_{\a,\b}$ is obtained from $S^3$ by taking connected sums with copies
%of $S^1 \times S^2$, so the knot defined by $w$ and $z$ in $M_{\a,\b}$ satisfies the
%hypothesis of Proposition~\ref{prop:doublyperiodic}.
%It follows that $n_z(\P_{\a,\b}) = n_w(\P_{\a,\b})$;
%and analogously, $n_z(\P_{\a,\d}) = n_w(\P_{\a,\d})$.

Since $(\S,\bolda,\boldb,\boldd)$ is subordinate to a bouquet for~$\L$,
each diagram $(\S,\bolda,\boldb,w,z)$, $(\S,\boldb,\boldd,w,z)$, and $(\S,\bolda,\boldd,w,z)$
defines a null-homologous knot in a connected sum of a number of copies of $S^1 \times S^2$.
Hence, by Proposition~\ref{prop:doublyperiodic},
$n_z(P) = n_w(P)$ for every $P \in \{\P_{\a,\b},\P_{\b,\d},\P_{\a,\d}\}$.
So $n_z(\P) = n_w(\P)$, and
\[
n_z(\psi_1) - n_w(\psi_1) = n_z(\psi_2) - n_w(\psi_2).
\]
Therefore, $c$ is independent of the triangle $\psi$ for fixed~$\x$ and~$\y$, cf.~Equation~\eqref{eq:c}.

To check that $c$ is independent of $\x$, consider another generator $\x'$
such that $\s(\x') = \s|_{Y_{\a,\b}} = \s(\x)$.
Since $\x$ and $\x'$ represent the same $\spinc$ structure,
%when considered as generators of $\CFh(Y_{\a,\b})$,
there is a Whitney disk $\phi \in \pi_2(\x',\x)$ (that possibly
crosses the basepoints $w$ and $z$).
If $\psi \in \pi_2(\x,\theta,\y)$, then $\phi \# \psi \in \pi_2(\x',\theta,\y)$.
By Lemma~\ref{lem:Alexanderdomain}, the
number $c$ defined in Equation~\eqref{eq:c} is the same for $\psi$ and $\phi\#\psi$,
so $c$ does not depend on $\x$. A similar reasoning also proves that $c$ is
independent of $\y$.

What remains to prove is that $c = 0$. We do
this by constructing a Whitney triangle $\psi$ for which $c = 0$.

\begin{figure}[t]
\begin{center}
\includegraphics{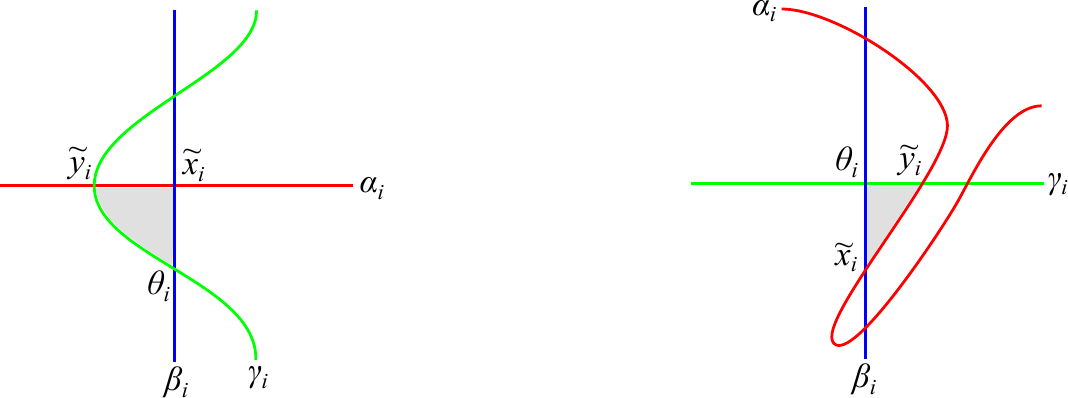}
\caption{This show the domain of the Whitney triangle $\td\psi$.
For $i \in \{\ell+1, \ldots, d\}$, the curves $\beta_i$ and $\delta_i$ are
small isotopic translates of each other, and -- after isotoping $\a_i$ -- we can find a ``small''
triangle bounded by $\a_i$, $\b_i$, and $\d_i$, shown shaded on the left.
For $i \in \{1, \ldots, \ell\}$, after applying finger moves to the $\a$-curves,
we can assume that there is a triangle, shown shaded on the right.
The sum of all these triangles is the domain of~$\td\psi$.}
\label{fig:tdpsi}
\end{center}
\end{figure}

By isotoping the $\a$-curves, we can create intersection points $\td\x \in \T_\a \cap \T_\b$
and $\td\y \in \T_\a \cap \T_\d$ such that there is a ``small'' triangle
$\td\psi \in \pi_2(\td\x,\theta,\td\y)$.
The domain of~$\td\psi$ is shown in Figure~\ref{fig:tdpsi}.
For each $i \in \{\ell+1,\ldots,d\}$, we isotope $\a_i$ -- pushing the other
$\a$-curves alongside -- until it intersects both $\delta_i$ and $\beta_i$ near $\theta_i$,
and consider the shaded triangle shown on the left-hand side of the figure.
For each $i \in \{1,\ldots,\ell\}$, after
some finger moves on the $\a$-curves -- again, pushing the other $\a$-curves
along -- we can assume that there is a
small triangle near each intersection point $\theta_i = \beta_i \cap \delta_i$,
as shown shaded on the right-hand side of the figure.
The sum of all these small triangles is the domain of the Whitney triangle~$\td\psi$.
We denote by $\td\x\in\T_{\a,\b}$ and $\td\y\in\T_{\a,\d}$
the generators connected by $\td\psi$; i.e., $\td\psi\in\pi_2(\td\x,\theta,\td\y)$.

The Whitney triangle $\td\psi$ satisfies $n_z(\td\psi) = n_w(\td\psi) = 0$,
but the constant $c$ is not necessarily defined for it, because $\s(\td\x)$
and $\s(\td\y)$ might not coincide with $\s|_{Y_{\a,\b}}$ and $\s|_{Y_{\a,\d}}$,
respectively, where~$\s \in \spinc(X)$ is the unique $\spinc$ structure,
cf.~Equation~\eqref{eq:c}.
The next lemma proves that we can replace $\td\psi$ with a Whitney triangle
$\psi$ for which the constant $c$ is defined.

\begin{lem}
\label{lem:replacepsi}
We can further isotope the $\a$-curves such that
there is a Whitney triangle $\psi \in \pi_2(\x, \theta, \y)$ satisfying
\begin{itemize}
\item{$n_z(\psi) = n_w(\psi) = 0$ and}
\item{$\s(\x) = \s|_{Y_{\a,\b}}$ and $\s(\y) = \s|_{Y_{\a,\d}}$.}
\end{itemize}
\end{lem}
\begin{proof}
Given generators $\x' = (x'_1, \ldots, x'_d)$ and
$\x'' = (x''_1, \ldots, x''_d)$ in $\T_\a \cap \T_\b$,
Ozsv\'ath and Szab\'o associate to them \cite[Definition 2.11]{OSz}
a class $\e(\x',\x'') \in H_1(Y_{\a,\b})$.
Choose 1-chains $a \subset \bolda$ and $b \subset \boldb$
such that
\[\de a = \de b = x''_1 + \cdots + x''_d - x'_1 - \cdots - x'_d.\]
Then $a-b$ represents an element of $H_1(\Sigma)$ whose image
in $H_1(Y_{\a,\b})$ under the inclusion map is $\e(\x',\x'')$.
Ozsv\'ath and Szab\'o proved \cite[Lemma 2.19]{OSz} that
\begin{equation}
\label{eq:eps}
\s(\x'') - \s(\x') = \PD(\e(\x',\x'')).
\end{equation}

Consider the Whitney triangle $\td\psi \in \pi_2(\td\x,\theta,\td\y)$
defined above, and whose domain is shown in Figure~\ref{fig:tdpsi}.
Its domain is the disjoint union of $d$ triangles $\td T_1, \dots, \td T_d$.

We define the homology classes $h_1 \in H_1(Y_{\a,\b})$ and $h_2 \in H_1(Y_{\a,\d})$ as
\begin{subequations}
\label{eq:defh}
\begin{align}
h_1 &= \PD \left(\s|_{Y_{\a,\b}} - \s(\td\x)\right) \text{ and}\\
h_2 &= \PD \left(\s|_{Y_{\a,\d}} - \s(\td\y)\right),
\end{align}
\end{subequations}
where $\s$ is the unique $\spinc$ structure on $X$.
By Lemma~\ref{lem:H1Sigma}, there is a homology class $h \in H_1(\Sigma)$
such that $i(h) = (h_1, h_2)$; i.e., $h$ represents $h_1$ in $H_1(Y_{\a,\b})$
and $h_2$ in $H_1(Y_{\a,\d})$.
We can represent $h$ as~$m\lambda$, where $\lambda$ is a simple closed curve
on $\Sigma$ that satisfies the following conditions:
\begin{itemize}
\item{$\lambda$ intersects the triangle $\td T_1$ as on the left-hand side of Figure~\ref{fig:psi},}
\item{$\lambda$ is disjoint from all the triangles $\td T_2, \ldots, \td T_d$, and}
\item{$\lambda$ is disjoint from the basepoints $z$ and $w$.}
\end{itemize}
\begin{figure}[t]
\begin{center}
\includegraphics{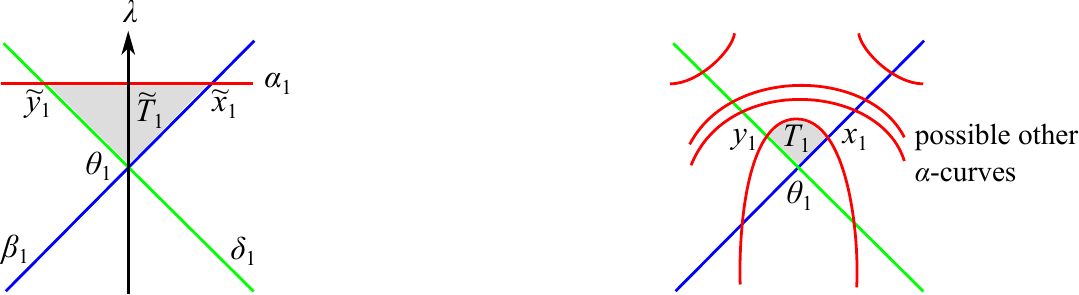}
\caption{The picture above shows how to modify the Whitney triangle
$\td\psi$ defined in Figure~\ref{fig:tdpsi} to obtain a Whitney triangle
$\psi$ satisfying the requirements of Lemma~\ref{lem:replacepsi}.
The picture on the left shows the loop $\lambda$ near the triangle $\td T_1$.
The picture on the right shows the new triangle $T_1$ in the triple Heegaard
diagram obtained after performing a finger move on the $\a$-curves
along $\lambda$.}
\label{fig:psi}
\end{center}
\end{figure}
If we perform a finger move on the $\a$-curves along the loop~$m\lambda$,
the result will look like the right-hand side of Figure~\ref{fig:psi}.
If $x_1$ and $y_1$ are as on the right-hand side of Figure~\ref{fig:psi},
we define $\x=(x_1,\td x_2, \ldots, \td x_d)$ and $\y=(y_1, \td y_2, \ldots, \td y_d)$.
Notice that, by construction,
\begin{equation}
\label{eq:epsh}
\e(\td\x, \x) = h_1 \text{ and } \e(\td\y, \y) = h_2.
\end{equation}
Let $\psi$ be a Whitney triangle with domain
$T_1 \sqcup \td T_2 \sqcup \cdots \sqcup \td T_d$, where $T_1$ is the
shaded triangle on the right-hand side of Figure~\ref{fig:psi}. By construction,
we have $n_z(\psi) = n_w(\psi) = 0$. Furthermore, due to Equations~\eqref{eq:eps}, \eqref{eq:epsh},
and~\eqref{eq:defh}, we have
\begin{align*}
\s(\x) &= \s(\td\x) + \PD(\e(\td\x,\x))\\
&= \s(\td\x) + \PD(h_1)\\
&= \s(\td\x) + (\s|_{Y_\ab} - \s(\td\x)) = \s|_{Y_{\a,\b}}.
\end{align*}
Analogously, we have $\s(\y) = \s|_{Y_{\a,\d}}$.
\end{proof}

Before showing that $c = 0$ for the triangle $\psi \in \pi_2(\x,\theta,\y)$ constructed above,
we prove that the relative $\spinc$ structure $\rs(\psi) \in \spinc(\W_2)$ extends
to a relative $\spinc$ structure on $\W$.

Recall that $Y_1 = Y_{\a,\b}$ is obtained from $Y_0$ by performing surgery along some
framed 0-spheres. The belt circles of the 1-handles involved give rise to
embedded 2-spheres $O_1, \dots, O_p \subset Y_1$. Similarly,
$Y_2 = Y_{\a,\d}$ is obtained from $Y_3$ by surgery along some framed 0-spheres,
giving rise to embedded spheres $O_1',\dots,O_s' \subset Y_2$.
In Lemma~\ref{lem:replacepsi}, we achieved that $\s(\x) = \s|_{Y_{\a,\b}}$
and $\s(\y) = \s|_{Y_{\a,\d}}$. This implies that $\s(\x)$ extends
to $X_1$, or equivalently, that $\langle\, c_1(\s(\x)), [O_i] \,\rangle = 0$
for every $i \in \{1,\dots,p\}$. Similarly, $\langle\, c_1(\s(\y)), [O_j'] \,\rangle = 0$
for every $j \in \{1,\dots,s\}$. However
\[
\langle\, c_1(\s(\x)), [O_i] \,\rangle = \langle\, c_1(\rs(\x)), [O_i] \,\rangle = \langle\, c_1(\rs(\x),t), [O_i] \,\rangle,
\]
as $\rs(\x)$ and $\s(\x)$ are represented by the same vector field on~$M_1 \subset Y_1$.
Since $M_0$ obtained from $M_1$ by compressing the 2-spheres $O_1,\dots,O_p$, the equality
\[
\langle\, c_1(\rs(\x),t), [O_i] \,\rangle = 0
\]
implies that $\rs(\x)$ extends to $\rs_1 \in \spinc(\W_1)$.
Similarly, $\rs(\y)$ extends to $\rs_3 \in \spinc(\W_3)$. The Mayer-Vietoris sequence now implies
that there is a $\spinc$ structure $\rs \in \spinc(\W)$ such that $\rs|_{\W_1} = \rs_1$,
$\rs|_{\W_2} = \rs(\psi)$, and $\rs|_{\W_3} = \rs_3$.

We are now ready to prove that, for the Whitney triangle $\psi$ constructed above, $c=0$.
Recall that, by definition,
\[
c = \A_{S_2}(\y) - \A_{S_1}(\x)  = \langle c_2(\rs(\y), t), [S_2] \rangle - \langle c_1(\rs(\x), t), [S_1] \rangle,
\]
cf.~Equation~\eqref{eq:c}, Definition~\ref{def:Alexandergrading}, and Remark~\ref{rem:Alexandergrading}.
Since $\psi$ is a Whitney triangle connecting $\x$, $\theta$, and $\y$, we have
that $\rs(\psi)|_{M_1} = \rs(\x)$ and $\rs(\psi)|_{M_2} = \rs(\y)$, and therefore
\[
c =  \langle c_1(\rs(\psi), t), [S_2] \rangle - \langle c_1(\rs(\psi), t), [S_1] \rangle.
\]
Notice that we can omit the restrictions of the (relative) $\spinc$ structures
by the naturality of Chern classes.

Now the relative $\spinc$ structure $\rs(\psi)$ extends to some relative $\spinc$
structure $\rs \in \spinc(\W)$ by the above. Then, by
Lemmas~\ref{lem:1handles} and~\ref{lem:3handles}, we have
\[
c = \langle c_1(\rs, t), [S_3] \rangle - \langle c_1(\rs, t), [S_0] \rangle.
\]
From Lemma \ref{lem:samespinc}, it finally follows that $c=0$.

We can now conclude the proof of Lemma~\ref{lem:2handles1}.
By Equation~\eqref{eq:c}, for any Whitney triangle $\psi \in \pi_2(\x,\theta,\y)$,
where $\x \in \T_\a \cap \T_\b$ and $\y \in \T_\a \cap \T_\d$ are such that $\s(\x)$ and
$\s(\y)$ extend to $X_1$ and $X_3$, respectively, we have
\[
\A_{S_2}(\y) - \A_{S_1}(\x) + n_z(\psi) - n_w(\psi) = 0.
\]
If $\psi$ contributes to the surgery map~$f_{\H^1, \L}(\x)$,
then $n_w(\psi) = 0$, and it has a pseudo-holomorphic representative, so $n_z(\psi) \ge 0$.
Consequently, $A_{S_2}(\y) \le A_{S_1}(\x)$, as desired.
\qed

\subsection{Naturality maps}
\label{sec:naturality}

Recall that~\cite{naturality}, given two admissible Heegaard diagrams $\H$ and $\H'$ for the same
3-manifold $Y$, there is a naturality map
\[
f_{\H,\H'} \colon \CFh(\H) \To \CFh(\H'),
\]
which is the composition of maps associated to isotopies of the attaching sets, handleslides,
(de)sta\-bi\-li\-sations, and diffeomorphisms of the Heegaard surface isotopic to the
identity in~$Y$. On the homology, it induces
an isomorphism
\[
F_{\H,\H'} \colon \HFh(\H) \to \HFh(\H')
\]
that is independent of the sequence of Heegaard moves.

In our case, $\H$ and $\H'$ are doubly pointed Heegaard diagrams, which
define the same decorated knot $(Y, K, P)$.
Together with Dylan Thurston, the first author proved \cite[Proposition~2.37]{naturality}
that $\H$ and $\H'$ can be connected by a sequence of Heegaard moves
\emph{that do not cross the basepoints $w$ and $z$}.
If we forget about the $z$ basepoint,
this sequence induces the naturality map $f_{\H,\H'} \colon \CFh(\H) \to \CFh(\H')$ above.
As we explained, the $z$ basepoints on $\H$ and $\H'$ induce
filtrations on $\CFh(\H)$ and $\CFh(\H')$. It follows from the work of Ozsv\'ath-Szab\'o~\cite{OSz3}
and Rasmussen~\cite{Ras} that, if $f_{\H,\H'}$ is the map
associated to either an isotopy, a handleslide, a (de)stabilization, or a diffeomorphism
of the Heegaard surface isotopic to the identity in~$Y$, then
it preserves the knot filtration.
If $f_{\H,\H'}$ is an isotopy map or a handleslide map, then the map induced
on the $E^1$ page is the corresponding naturality map $F_{\H, \H'}$ on $\HFKh$;
i.e., it is the map obtained by counting all holomorphic triangles that
do not cross~$z$. If $f_{\H,\H'}$ is a (de)stabilization or diffeomorphism map, then it is
an isomorphism of filtered complexes.

As the above result is only outlined
in the works of Ozsv\'ath and Szab\'o~\cite{OSz3} and Rasmussen~\cite{Ras},
we provide a bit more detail.
With the techniques of this paper, we can prove the following analogue
of Lemma~\ref{lem:2handles1}.

\begin{lem}
\label{lem:naturality}
Let $K$ be a null-homologous knot in
$Y = \#_{i=1}^p (S^1 \times S^2).$
Choose a Seifert surface $S$ for $K$.
Suppose that $\H$ and $\H'$ are admissible doubly pointed Heegaard diagrams for $(Y, K, P)$
that only differ by an isotopy or a handleslide.

Given an admissible doubly pointed triple diagram $(\S,\bolda,\boldb,\boldd,w,z)$
for the Heegaard move $\H \to \H'$, if $\x \in \T_\a \cap \T_\b$, then
for any $\y \in \T_\a \cap \T_\d$ that has non-trivial coefficient
in the expansion of $f_{\H, \H'}(\x)$, we have that
\[
\F_{S}(\y) \le \F_{S}(\x).
\]
Furthermore, if $\psi$ is a holomorphic triangle connecting $\x \in \T_\a \cap \T_\b$, $\theta \in \T_\b \cap \T_\d$
(the top-dimensional generator of $\CFh(\S,\boldb,\boldd,w,z)$),
and $\y \in \T_\a \cap \T_\d$ that does not cross $w$, then
\begin{equation} \label{eqn:filtnat}
\F_{S}(\y) = \F_{S}(\x) - n_z(\psi).
\end{equation}
\end{lem}

\begin{rem}
The (de)stabilization and diffeomorphism maps do not appear in Lemma~\ref{lem:naturality}
because they are not triangle maps.
They are already isomorphisms at the level of filtered chain complexes.
\end{rem}

\begin{idea}
As in the proof of Lemma~\ref{lem:2handles1}, we let
\[
c = \A_S(\y) - \A_S(\x) - n_w(\psi) + n_z(\psi),
\]
and prove that this is independent of $\psi$, $\x$, and $\y$.
The main differences from the proof of Lemma~\ref{lem:2handles1}
are the following:
\begin{itemize}
\item{\textbf{Triply periodic domains:} We closely follow the
proof of Proposition~\ref{prop:triplyperiodic}. In this case,
$X \cong Y \times I$, the boundary of the 4-manifold $X_\triangle$
consists of $Y \sqcup Y \sqcup Y_{\boldb,\boldd}$, and the cobordisms
$X_1$ and $X_3$ are replaced by identity cobordisms $Y \times I$.
Finally, the proof of the injectivity of $\varphi$ follows from the
surjectivity of the map
\[
\rho \colon H_2(Y \times I) \oplus H_2(Y \times I) \oplus H_2(\hat X) \To H_2(X),
\]
as noted in Remark~\ref{rem:aftertp}.
}
\item{\textbf{Doubly periodic domains:} One can use
Proposition~\ref{prop:doublyperiodic} for the two copies of $Y$
and for $Y_{\boldb,\boldd}$.}
\item{\textbf{Proving that $c=0$:} This is easier
than in the case of the $2$-handle maps, because we already know
that the naturality map preserves the graded Euler characteristic,
and this forces the grading shift $c$ to be $0$. Also, as $X_1$ and $X_3$
are products, $\spinc$ structures automatically extend to them, hence we do not
need to isotope the $\a$-curves.}
\qed
\end{itemize}
\end{idea}

\subsection{Proof of Theorem~\ref{thm:main1}}
We are now ready to prove Theorem~\ref{thm:main1}.
In the proof we use the notation introduced in Section~\ref{sec:notation},
and we assume that the gluing map is the identity map, as explained in
Section~\ref{sec:gluingmap}.

Suppose that $\x$ is a generator of $\CFh(\H^0)$ such that $f_\CC(\x) \neq 0$.
Let $\y$ be a generator of $\CFh(\H^2_{\mb S})$ that appears in the expression
of $f_\CC(\x)$ with non-zero coefficient. Then there exist generators $\x' \in \CFh(\H^0_\bP)$,
$\x'' \in \CFh(\H^1)$, $\y'' \in \CFh(\H^1_\L)$, and $\y' \in \CFh(\H^2)$ that appear with
non-zero coefficient in $f_{\H^0, \bP}(\x)$, $f_{\H^0_\bP, \H_1}(\x')$, $f_{\H^1, \L}(\x'')$, and
$f_{\H^1_\L, \H^2}(\y'')$, respectively, and such that
$\y$ appears with non-zero coefficient in $f_{\H^2, \bS}(\y')$.
Notice that, by construction, $\s(\x'')$ extends to~$X_1$ and $\s(\y'')$ extends to~$X_3$.

By Lemma~\ref{lem:naturality}, we know that the naturality maps preserve the knot filtration,
and by Corollaries~\ref{cor:1handles} and~\ref{cor:3handles} so do the
maps $f_{\H^0, \bP}$ and $f_{\H^2, \bS}$. Finally, Lemma~\ref{lem:2handles1}
proves that $\F_{S_2}(\y'') \leq \F_{S_1}(\x'')$. By putting all these together,
we obtain that
\begin{equation}
\label{eqn:preservefiltration}
\F_{S_3}(\y) = \F_{S_2}(\y') \leq \F_{S_2}(\y'') \leq \F_{S_1}(\x'') \leq \F_{S_1}(\x') = \F_{S_0}(\x).
\end{equation}
Thus $f_\CC$ is a map of filtered complexes and so, by Remark~\ref{rem:morphss},
it induces a morphism of spectral sequences.

Furthermore, each of the maps $f_{\H^0, \bP}$, $f_{\H^0_\bP, \H^1}$,
$f_{\H^1, \L}$, $f_{\H^1_\L, \H^2}$, and $f_{\H^2, \bS}$ is a map
of filtered complexes. The map induced by $f_\CC$ on the $E^1$ page
is the composition of the maps induced by each of the above maps on the $E^1$ page.

We now consider the case when the inequalities in
Equation~\eqref{eqn:preservefiltration} are all equalities.
Lemmas~\ref{lem:1handles} and~\ref{lem:3handles} imply that the maps
induced by $f_{\H^0, \bP}$ and $f_{\H^2, \bS}$ on the $E^1$ page
are the 1- and 3-handle maps for $\HFKh$. As for the 2-handle map $f_{\H^1, \L}$,
by Equation~\eqref{eq:shift} in Lemma~\ref{lem:2handles1}, we have that $\F(\y'') = \F(\x'')$
if and only if there is a pseudo-holomorphic triangle $\psi$ connecting
$\x''$, $\theta$, and $\y''$ such that $n_w(\psi) = n_z(\psi) = 0$,
and in this case all such holomorphic triangles satisfy this equality.
Hence, the map induced by $f_{\H^1, \L}$ on the $E^1$ page is the 2-handle
map for $\HFKh$.
Finally, it follows from the discussion in Section~\ref{sec:naturality}
that the maps induced on the $E^1$ page by the naturality maps for $\CFh$
are the naturality maps for $\HFKh$. Alternatively, one can use Equation~\eqref{eqn:filtnat}
in Lemma~\ref{lem:naturality} and argue in the same way as for the 2-handle
maps.

This immediately implies that the map induced by $f_\CC$ on the $E^1$ page is
obtained by counting (for the naturality maps and the 2-handle map)
the pseudo-holomorphic triangles that do not cross $w$ and $z$, and so it is $F_\CC$.

On the other hand, the map induced by $f_\CC$ on the total homology
is given by counting all holomorphic triangles that do not cross $w$
but might cross $z$. This is precisely the map~$\hat{F}_X \colon \HFh(S^3) \to \HFh(S^3)$
induced by the cobordism $X$.
Since $H_1(X) = H_2(X) = 0$, the map $\hat{F}_X = \Id_{\HFh(S^3)}$ by~\cite[Lemma~3.4]{OSz14}.
\qed

%%%%%%%%%%%%%%%%%%%%%%%%%%%%%%%%%%%%%%%%%%%%%%%%%%%%%%%%%%%%%%%%%%%%%%%%%%

\section{Concordance maps preserve the homological grading} \label{sec:homol}

In this section, we show that concordance maps also behave well
with respect to another grading of $\CFh$, namely the homological
grading.

Let $\H$ be an admissible pointed Heegaard diagram for the closed, connected, oriented,
based 3-manifold~$(Y,w)$, together with a $\spinc$ structure $\s \in \spinc(Y)$
such that $c_1(\s) \in H^2(Y)$ is torsion.
Ozsv\'ath and Szab\'o~\cite[Section~4]{OSz} showed that
$\CFh(\H, \s)$ admits a relative $\Z$-grading.
For generators $\x$, $\y \in \CFh(\H, \s)$ and $\phi \in \pi_2(\x,\y)$, we have
\begin{equation}
\label{eq:Maslovgrading}
\gr(\x,\y) = \mu(\phi) - 2 n_w(\phi).
\end{equation}
They showed~\cite[Theorem~7.1]{OSz10} that this can be lifted to an absolute $\Q$-grading~$\agr$,
in the sense that $\gr(\x,\y) = \agr(\x) - \agr(\y)$.
Such a grading is called the \emph{Maslov grading} or \emph{homological grading}.

\begin{ex}
\label{ex:S3}
If $Y = S^3$ with its unique $\spinc$ structure~$\s_0$, and if $\H$ is a Heegaard diagram of~$Y$,
then on $\CFh(\H,\s_0)$ the absolute $\Q$-grading is actually an absolute $\Z$-grading.
The generator of $\HFh(S^3,\s_0) \cong \Z_2$ is homogeneous of grading zero.

More generally, if $Y = \#_{i=1}^k (S^1 \times S^2)$ with Heegaard diagram~$\H$,
and $\s_0 \in \spinc(Y)$ is such that $c_1(\s_0) = 0$,
then $\agr$ is an absolute $\Z$-grading on $\CFh(\H,\s_0)$.
\end{ex}

The main result of this section is the following.

\begin{thm}
\label{thm:homologicalgrading}
Let $\CC$ be a decorated concordance from $(S^3, K_0, P_0)$ to $(S^3, K_1, P_1)$,
and let $\H_i$ be an admissible doubly pointed diagram of $(S^3, K_i, P_i)$
for $i \in \{0,1\}$. Then, the chain map
\[
f_\CC \colon \CFh(\H_0) \To \CFh(\H_1)
\]
preserves the absolute homological grading; that is, if $x \in \CFh(\H_0)$ is $\agr$-ho\-mo\-ge\-neous,
so is $f_\CC(x)$, and if $f_\CC(x) \neq 0$, then
\[
\agr(f_\CC(x)) = \agr(x).
\]
\end{thm}

\begin{rem}
Notice that the statement of Theorem~\ref{thm:homologicalgrading}
is stronger than the fact that $f_\CC$ preserves the Maslov filtration.
We actually claim that the Maslov grading is not decreased by $f_\CC$.
\end{rem}

\begin{idea}
We proceed similarly to the proof of Theorem~\ref{thm:main1},
and use the notation from Section~\ref{sec:notation} and Figure~\ref{fig:notation}.
As the diffeomorphism $D$ constructed in Section~\ref{sec:gluingmap}
induces a homomorphism $D_*$ that preserves the homological grading,
we can assume the gluing map is trivial and we are dealing with a special cobordism.

First, we prove that, in the right $\spinc$ structure, each map $f_{\H^0, \bP}$, $f_{\H^0_\bP, \H^1}$,
$f_{\H^1, \L}$, $f_{\H^1_\L, \H^2}$, and $f_{\H^2, \bS}$
preserves the relative Maslov grading~$\gr$.
This is only implicit in the work of Ozsv\'ath and Szab\'o~\cite{OSz10},
so we provide more detail.
Then we show that the absolute grading shift of $f_\CC$,
which is the composition of all the above maps, is zero.

For the 1- and 3-handle maps $f_{\H^0, \bP}$ and $f_{\H^2, \bS}$,
it is straightforward to check that the relative Maslov
grading is preserved using Equation~\eqref{eq:Maslovgrading} above.

Now consider the 2-handle map $f_{\H^1, \L}$.
Let $(\S,\bolda,\boldb,\boldd,w,z)$ be an admissible triple Heegaard diagram subordinate
to a bouquet for $\L$.
For generators $\x \in \T_\a \cap \T_\b$ and $\y \in \T_\a \cap \T_\d$
such that $\s(\x) = \s|_{Y_{\a,\b}}$ and $\s(\y)=\s|_{Y_{\a,\d}}$, where $\s$
denotes the unique $\spinc$ structure on $X$, and
for every Whitney triangle $\psi \in \pi_2(\x, \theta, \y)$,
we let
\[
d = \agr(\y) - \agr(\x) + \mu(\psi) - 2 n_w(\psi).
\]
We show that $d$ is independent of $\psi$, $\x$, and $\y$.
Since the triangles $\psi$ contributing to $f_{\H^1, \L}$ have
$\mu(\psi)=0$ and $n_w(\psi)=0$, it follows that the absolute grading
is shifted by $d$, so the relative grading is preserved.

We already know from the work of Ozsv\'ath and Szab\'o~\cite{OSz} that
the naturality maps $f_{\H^0_\bP, \H^1}$ and $f_{\H^1_\L, \H^2}$
preserve the relative homological grading~$\gr$.
Alternatively, this can also be shown using the techniques
of Section~\ref{sec:naturality}.

Finally, $f_\CC$, which is the composition of
all the above maps, preserves the relative homological grading,
or equivalently, it shifts the absolute homological grading by
some constant~$e$. This implies that, for every $r \in \NN$, the map $E^r(f_\CC)$ shifts the homological
grading by the same constant $e$ independent of~$r$. Since we know that
the map in total homology is $\Id_{\HFh(S^3)}$ and preserves the absolute
grading by \cite[Lemma~3.4]{OSz14}, it immediately follows that $e=0$.
\qed
\end{idea}

The rest of this section is devoted to filling in the details of the above outline.

\subsection{$\spinc$ structures}
\label{sec:hgetridspinc}

Let $\s$ be the unique $\spinc$ structure on $X$. Then
\[
f_\CC = f_{\CC, \s} = f_{\H^2, \bS, \s} \circ \cdots \circ f_{\H^0, \bP, \s},
\]
where the restrictions of $\s$ are omitted for the sake of clarity.

So it suffices to consider the above maps in the $\spinc$ structure~$\s$. In the rest of the
section, we will focus on the maps $f_{\H^2, \bS, \s}, \ldots, f_{\H^0, \bP, \s}$,
and for simplicity, we will denote the restrictions of~$\s$ by the same letter.

\subsection{1- and 3-handles}

The 1-handle map $f_{\H^0, \bP,\s}$ satisfies the following.

\begin{lem}
\label{lem:h1handles}
Let $\x''$, $\td\x'' \in \CFh(\H^0, \s)$ be generators. Then
\[
\gr(\x'',\td\x'') = \gr(f_{\H^0, \bP,\s}(\x''), f_{\H^0, \bP,\s}(\td\x''));
\]
i.e., the relative homological grading is preserved under the 1-handle map.
\end{lem}

\begin{proof}
Let $\phi \in \pi_2(\x'',\td\x'')$. Then the domain of $\phi$ also represents
a Whitney disk between $f_{\H^0, \bP}(\x'')$ and
$f_{\H^0, \bP}(\td\x'')$ in the Heegaard diagram $\H^0_\bP$ that we also denote by~$\phi$.
By Equation~\eqref{eq:Maslovgrading}, we have
\[
\gr(\x'',\td\x'') = \mu(\phi) - 2 n_w(\phi) = \gr(f_{\H^0, \bP,\s}(\x''),f_{\H^0, \bP,\s}(\td\x'')).
\qedhere
\]
\end{proof}

A dual argument gives the following result for the 3-handle map $f_{\H^2, \bS,\s}$.

\begin{lem}
\label{lem:h3handles}
Let $\y'$, $\td\y' \in \CFh(\H^2, \s)$ be generators such that $f_{\H^2, \bS,\s}(\y') \neq 0$
and $f_{\H^2, \bS,\s}(\td\y') \neq 0$. Then
\[
\gr(\y',\td\y') = \gr(f_{\H^2, \bS,\s}(\y'),f_{\H^2, \bS,\s}(\td\y'));
\]
i.e., the relative homological grading is preserved under the 3-handle map.
\end{lem}

\subsection{2-handles}

For 2-handles, we have the following.

\begin{lem}
\label{lem:h2handles}
Let $\x$, $\td\x \in \CFh(\H^1)$ be generators such that $\s(\x) = \s(\td\x) = \s$.
Then $f_{\H^1, \L,\s}(\x)$ and $f_{\H^1, \L,\s}(\td\x)$ are $\agr$-homogeneous, and if they are
non-zero, then
\[
\gr(\x,\td\x) = \gr(f_{\H^1, \L,\s}(\x), f_{\H^1, \L,\s}(\td\x)).
\]
\end{lem}

\begin{proof}
For $\x \in \CFh(\H^1)$, $\y \in \CFh(\H^1_\L)$, and $\psi \in \pi_2(\x, \theta, \y)$
such that $\s(\psi) = \s$, let
\begin{equation}
\label{eq:defd}
d = \agr(\y) - \agr(\x) + \mu(\psi) - 2 n_w(\psi).
\end{equation}

First, we check that $d$ is independent of $\psi$. As in the proof of
Lemma~\ref{lem:2handles1}, it suffices to show that, for every triply periodic domain $\mc P$,
\begin{equation}
\label{eq:minw}
\mu(\mc P) = 2 n_w(\mc P).
\end{equation}
Since every triply periodic domain is the sum of doubly periodic domains
by Proposition~\ref{prop:triplyperiodic}, it is sufficient to prove
Equation~\eqref{eq:minw} in the case of doubly periodic domains in Heegaard
diagrams of $Y_{\a,\b}$, $Y_{\a,\d}$, and $Y_{\b,\d}$.

Consider, for example, $Y_{\a,\b}$ and $\bz \in \T_\a \cap \T_\b$,
with a periodic domain $\P \in \Pi_{\a,\b}$ based at $\bz$.
As $\s(\bz)$ extends to the cobordism $X_1$, we see that $c_1(\s(\bz))$
vanishes on the belt spheres of the 1-handles.
Furthermore, since $H(\P) \in H_2(Y)$ is a linear combination of the belt spheres,
we obtain that
\[
\langle\, c_1(\s(\bz)), H(\P) \,\rangle = 0.
\]
By the work of Ozsv\'ath and Szab\'o \cite[Theorem~4.9]{OSz} and Lipshitz~\cite[Lemma~4.10]{Lipshitz},
\[
\mu(\P) = \langle\, c_1(\s(\bz)), H(\P) \,\rangle + 2n_w(\P),
\]
and the result follows.
This proves that $d$ is independent of $\psi$.

Next, we check that $d$ is independent of $\x$ and $\y$.
Let $\td\x$ be another generator of $\CFh(\H^1)$ such that
$\s(\td\x) = \s$. Then there is a Whitney disk $\phi \in \pi_2(\td\x,\x)$,
hence $\phi \# \psi \in \pi_2(\td\x,\theta,\y)$. Then, by Equation~\eqref{eq:Maslovgrading},
\[
\begin{split}
d & = \agr(\y) - \agr(\x) + \mu(\psi) - 2 n_w(\psi) \\
& = \agr(\y) - \agr(\x) + \mu(\psi) - 2 n_w(\psi) + (\agr(\x) - \agr(\td\x) + \mu(\phi) - 2 n_w(\phi)) \\
& = \agr(\y) - \agr(\td\x) + \mu(\phi \# \psi) - 2 n_w(\phi \# \psi).
\end{split}
\]
Thus, $d$ is independent of $\x$. An analogous argument shows independence of $\y$.

Finally, all the holomorphic triangles that appear in the definition
of the map $f_{\H^1, \L,\s}$ satisfy
$\mu(\psi) = 0$ and $n_w(\psi) = 0$.
Then, it follows from Equation~\eqref{eq:defd} that $f_{\H^1, \L,\s}$
increases the absolute grading $\agr$ by $d$. In particular, it preserves the
relative grading $\gr$.
\end{proof}

\subsection{Naturality maps}

We already know from the work of Ozsv\'ath and Szab\'o~\cite{OSz}
that the naturality maps preserve the Maslov grading.
Alternatively, one can prove that the handleslide and isotopy maps
preserve the Maslov grading using the techniques of Lemma~\ref{lem:h2handles}.
The (de)stabilization maps are already isomorphisms on the chain level.

\subsection{Proof of Theorem~\ref{thm:homologicalgrading}}

As explained in Section~\ref{sec:hgetridspinc},
\[
f_\CC = f_{\H^2, \bS, \s} \circ f_{\H^1_\L, \H^2, \s} \circ f_{\H^1, \L, \s} \circ f_{\H^0_\bP, \H^1, \s} \circ f_{\H^0, \bP, \s}.
\]
All the above maps preserve the relative Maslov grading by Lemmas~\ref{lem:h1handles},
\ref{lem:h3handles} and~\ref{lem:h2handles}, so $f_\CC$ shifts the absolute Maslov
grading by some constant~$e$.
It follows that the maps induced between the spectral sequences $E^r(f_\CC)$
shift the absolute Maslov grading by the same constant~$e$. On the other hand,
the map in total homology is $\Id_{\HFh(S^3)}$, which is homogeneous of degree $0$,
so we obtain that $e=0$.
\qed

% ----------------------------------------------------------------
\bibliographystyle{amsplain}
\bibliography{topology}
\end{document}